\documentclass[preprint,11pt]{elsarticle}

\usepackage{amssymb}
\usepackage{amsmath}

\usepackage{hyperref}
\usepackage{amsfonts}
\usepackage{graphicx}
\usepackage{epstopdf}
\usepackage{algorithmic}
\ifpdf
 \DeclareGraphicsExtensions{.eps,.pdf,.png,.jpg}
\else
 \DeclareGraphicsExtensions{.eps}
\fi

\usepackage{amsopn}
\usepackage{tikz}
\usepackage{mathrsfs}
\usepackage{subcaption} 

\definecolor{ForestGreen}{RGB}{34,139,34}

\usepackage{bm}
\usepackage{xspace}

\newcommand{\mr}[1]{\ensuremath{\mathrm{#1}}}

\newcommand{\fnc}[1]{\ensuremath{\mathit{#1}}}
\newcommand{\bfnc}[1]{\ensuremath{\bm{\mathit{#1}}}}

\newcommand{\U}[0]{\ensuremath{\bfnc{U}}}

\newcommand{\E}[0]{\ensuremath{\fnc{E}}}

\newtheorem{theorem}{Theorem}
\newtheorem{remark}{Remark}
\newproof{proof}{Proof}

\journal{Applied Numerical Mathematics}

\begin{document}

\begin{frontmatter}


\author{Mohammed Sayyari\corref{cor1}\fnref{label2}}
\ead{malsayya@odu.edu}
\author{Nail K. Yamaleev\fnref{label2}}
\ead{nyamalee@odu.edu}
\fntext[label2]{Department of Mathematics and Statistics, Old Dominion University, Norfolk, VA, USA}
\cortext[cor1]{Corresponding author.}

\title{Implicit dual time-stepping positivity-preserving entropy{-}stable schemes for the compressible Navier{--}Stokes equations} 



\begin{abstract}
We generalize the explicit high-order positivity-preserving entropy{-}stable spectral collocation schemes developed in {\cite{upperman2023first} and \cite{yamaleev2023high}} for the three-dimensional (3D) compressible Navier{--}Stokes equations to a time{-}implicit formulation. The time derivative terms are discretized by using the first- and second-order implicit backward difference formulas (BDF1 and BDF2) that are well suited for solving steady-state and time-dependent viscous flows at high Reynolds numbers, respectively. The nonlinear system of discrete equations at each physical time step is solved by using a dual time-stepping technique. The proposed scheme is provably entropy{-}stable and positivity-preserving and provides unconditional stability properties in the physical time. Numerical results demonstrating {the} accuracy and positivity-preserving properties of the new dual time-stepping scheme are presented for supersonic viscous flows with strong shock waves and contact discontinuities. 
\end{abstract}



\begin{keyword}
high-order {spectral collocation} methods \sep
dual time-stepping schemes \sep
summation-by-parts (SBP) operators \sep
entropy stability \sep
positivity-preserving methods \sep
Brenner--Navier--Stokes regularization

\MSC[2020] 
65L06 \sep
65M12 \sep
65M70 \sep
76J20 \sep
76M25 

\end{keyword}

\end{frontmatter}


\section{Introduction}
\label{sec:introduction}


It is well known that entropy stability provides {the foundation for $L_2$ stability of Navier--Stokes simulations.  However, the existence of an entropy pair and the symmetrization of the Navier{--}Stokes equations, which are both required for proving the  entropy stability,  can be guaranteed only} if the thermodynamic variables are strictly positive in the entire space-time domain.  {The entropy stability and positivity} properties are especially critical in simulations of high-Mach- and high-Reynolds-number flows as they exhibit extreme compressibility and high-temperature effects. These effects are exacerbated in high-order solvers, as unresolved flow features lead {to well-known} Gibbs oscillations, {which may violate 
the positivity of thermodynamic variables and the entropy inequality.}
 Thus, developing high-order entropy{-}stable and positivity-preserving schemes is critical for solving viscous super- and hypersonic flows described by the compressible Navier--Stokes equations.


{ The primary focus of this paper is the development of a numerical scheme of arbitrary spatial order of accuracy for the compressible Navier--Stokes equations that provides unconditional stability in physical time,  satisfies a discrete entropy inequality,  and preserves the positivity of thermodynamic variables.}
Though there is a vast body of literature on bound-preserving schemes for hyperbolic conservation laws and the Euler equations (e.g., see the following survey papers~\cite{shu2018bound,zhang2011survey} 
and references therein), papers on high-order positivity-preserving methods for the compressible Navier{--}Stokes equations are rare ~\cite{upperman2023first,yamaleev2023high,zhang2017positivity,guermond2021second,upperman2021high,upperman2022positivity,yamaleev2022positivity,lin2023positivity}. All high-order positivity-preserving methods mentioned above are explicit in time, which imposes {a} Courant-Friedrichs-Lewy (CFL)-type constraint on the time step. Since high-Mach- and high-Reynolds-number flows are characterized by the presence of large gradients, enforcing the positivity and entropy stability relies on {a balance between} artificial dissipation and bounds on the time step~\cite{upperman2023first,yamaleev2023high}. Consequentially, this leads to very stiff problems, which incur higher computational costs. Therefore, a natural progression is to turn to implicit schemes that remedy the stiffness and can eliminate the CFL-type time step constraint caused by the explicit time integrators.
 

While nonlinear implicit solvers {mitigate} the stiffness present in viscous flow simulations, they do not inherently preserve positivity at each nonlinear iteration and, {therefore}, are not provably positivity-preserving. Herein,  we develop a new dual time-stepping scheme that provides provably positivity-preserving and entropy stability properties for the discretized 3D compressible Navier{--}Stokes equations. Dual time-stepping schemes allow for converting any time-dependent problem into a steady-state problem at each physical time step~\cite{belov1995new,Knoll1998,pandya2003implementation}. In the present work,  we develop a new implicit (in the physical time) dual time-stepping method based on explicit (in the pseudotime) positivity-preserving and entropy{-}stable schemes {introduced} in~\cite{upperman2023first,yamaleev2023high,upperman2021high,upperman2022positivity,yamaleev2022positivity}. Though using explicit schemes in the pseudotime is not novel, e.g., see~\cite{nakashima2014development,nordstrom2019dual}, what is new, is that the proposed dual time-stepping schemes both preserve the positivity of thermodynamic variables and satisfy the entropy inequality. One of the challenges in using explicit schemes is that the pseudotime step is bounded by the CFL-type condition, thus limiting the overall efficiency of the corresponding dual time-stepping scheme~\cite{arnone1993multigrid}. However, the bound on the pseudotime step required for positivity of density for the proposed dual time-stepping schemes is shown to be, in fact, higher than that of the corresponding purely explicit counterpart. While acknowledging some possible advantages of using high-order strong stability preserving (SSP) explicit Runge-Kutta methods ({e.g., see} \cite{GKS}) for discretizing the pseudotime derivative, in the present analysis, we consider {only} a first-order explicit update in the pseudotime and explore the extent of its potential and limitations.  Our numerical results demonstrate the accuracy and positivity-preserving properties of the proposed dual time-stepping schemes for both steady and unsteady viscous flows with strong shock waves and contact discontinuities. 


The paper is organized as follows.  We begin by defining the regularized Navier--Stokes equations in Section \ref{sec:rns}. Then,  the baseline explicit first-order positivity-preserving entropy{-}stable scheme is {briefly described} in Section \ref{sec:semi-sidcrete}. In Sections \ref{sec:dual_time-stepping} and \ref{sec:positivity},  we construct a new entropy{-}stable dual time-stepping scheme based on the first-order implicit backward difference (BDF) time integrator and derive bounds on the pseudotime step size required for enforcing the positivity of thermodynamic variables. Then, this first-order dual time{-}stepping positivity-preserving entropy{-}stable BDF1 scheme {is generalized} to high-order spatial discretizations in Section~\ref{sec:dual-high}. In Section~\ref{sec:bdf2-positivity},  we extend the proposed methodology to the implicit second-order BDF2 scheme by deriving bounds on the pseudotime step that guarantee the positivity of the density and internal energy. Lastly,  we verify and demonstrate the proposed schemes' positivity-preserving and entropy stability properties  on some benchmark test problems in Section \ref{sec:results}.


\section{Regularized Navier{--}Stokes equations}
\label{sec:rns}


There are no theoretical results in the literature showing that the compressible Navier{--}Stokes equations themselves guarantee the positivity of the thermodynamic variables. { As noted in the previous section, without these positivity properties, 
it cannot be proven that the Navier{--}Stokes equations are stable in the entropy sense.  In contrast to the classical Navier{--}Stokes model, the Brenner-Navier{--}Stokes equations introduced in~\cite{brenner2005navier} guarantee the global-in-time positivity of thermodynamic variables and satisfy a large class of entropy inequalities and the minimum entropy principle ~\cite{feireisl2010new,guermond2014viscous}. }
The Brenner--Navier--Stokes equations in curvilinear coordinates $(\xi_1,\xi_2, \xi_3)$ are given by
\begin{align}
    \label{eq:BNS_C}
    \frac{\partial(J \U)}{\partial t}
    + \sum\limits_{d=1}^{3} \frac{\partial {\bm F}_{d}^{(I)}}{\partial \xi_d} 
    &= \sum\limits_{d=1}^{3}\frac{\partial {\bm F}_{d}^{(B)}}{\partial \xi_d},
    \quad \forall \left(\xi_1,\xi_2,\xi_3\right)\in\ \hat{\Omega},
    \quad t> 0, \\
    \nonumber {\bm F}_{d} &\equiv \sum\limits_{i=1}^3 J \frac{\partial \xi_d}{\partial x_i} {\bm F}_{i} \quad
    \ {\rm for} \ d=1,2,3,
\end{align}
where $\U = \left[\rho,\rho v_1,\rho v_2,\rho v_3,\rho\E\right]^\top$ is a vector of the conservative variables ($\rho$~- density, $\rho v_i$ - momentum, and $\rho E$ - total energy), $J=\left|\frac{\partial(x_1, x_2, x_3)}{\partial(\xi_1, \xi_2, \xi_3)} \right|$ is the metric Jacobian, and $(x_1, x_2, x_3)$ are the Cartesian coordinates. Note that the above formulation can be directly used for fully unstructured hexahedral grids, because each convex hexahedral grid cell is individually mapped onto a reference element $(\xi_1, \xi_2, \xi_3) \in \hat{\Omega}=[-1, 1]^3$ in the computational domain, such that each individual transformation is a diffeomorphism. 
The Brenner viscous flux, ${\bm F}_{d}^{(B)}$, is defined as follows:
\begin{equation}\label{eq:FB}
 {\bm F}_{d}^{(B)}
    = {\bm F}_{d}^{(V)}
 + \sigma \frac{\partial \rho}{\partial x_d}
    \begin{bmatrix}
      1 & \bf{v} & \E 
    \end{bmatrix}^\top,
\end{equation}
where $\sigma$ is the volume diffusivity, ${\bf v}=\left[v_1, v_2, v_3\right]^\top$ is 
the velocity vector, and $E$ is the specific total energy.  In Eqs. (\ref{eq:BNS_C}-\ref{eq:FB}), ${\bm F}_{d}^{(I)}$ and ${\bm F}_{d}^{(V)}$ are the inviscid and viscous fluxes of the conventional Navier{--}Stokes equations, which are given by
\begin{equation}\label{eq:FNS}
 {\bm F}_{d}^{(I)}
 = \begin{bmatrix}
  \rho v_d \\
  \rho v_1 v_d + P\delta_{1d} \\
  \rho v_2 v_d + P\delta_{2d} \\
  \rho v_3 v_d + P\delta_{3d} \\
  (\rho\E + P)v_d
 \end{bmatrix},\quad
 {\bm F}_{d}^{(V)}
 = \begin{bmatrix}
  0 \\
  \tau_{1d} \\
  \tau_{2d} \\
  \tau_{3d} \\
  \tau_{id}v_i 
  + \kappa \frac{\partial T}{\partial x_d}
 \end{bmatrix},
\end{equation}
where $\delta_{id}$ is the Kronecker delta, $P=\rho R_g T$ is the pressure, $R_g$ is the gas constant, $T$ is the temperature, and $\kappa$ is the heat conductivity. The stress tensor is given by
\begin{equation}
\nonumber
 \tau_{id} = \mu \left( \frac{\partial v_i}{\partial x_d} + \frac{\partial v_d}{\partial x_i} - \frac{2}{3}\delta_{id}\frac{\partial v_l}{\partial x_l}\right),
\end{equation}
where $\mu$ is the dynamic viscosity. For brevity, the Einstein notation is used where appropriate, such as in Eq~\eqref{eq:FNS}, $\tau_{id}v_i:=\sum_{i=1}^3 \tau_{id}v_i$. 

{
To achieve the positivity of the thermodynamic variables,  we regularized the Navier{--}Stokes system by adding artificial dissipation in the form of the diffusion operator of its Brenner counterpart given by Eqs. (\ref{eq:BNS_C}-\ref{eq:FB}), thus leading to}
\begin{equation}\label{rns3d}
 \frac{\partial(J \U)}{\partial t}
 + \sum\limits_{d=1}^{3} \frac{\partial {\bm F}_{d}^{(I)}}{\partial \xi_d} 
 = \sum\limits_{d=1}^{3} \left[
 \frac{\partial {\bm F}_{d}^{(V)}}{\partial \xi_d}
 + \frac{\partial {\bm F}_{d}^{(AD)}}{\partial \xi_d}
 \right],
\end{equation}
where the artificial dissipation flux ${\bm F}_{d}^{(AD)}$ can be obtained from the viscous flux of the Brenner{--}Navier{--}Stokes equations, ${\bm F}_{d}^{(B)}$, by setting $\mu=\mu^{AD}$, $\sigma = c_\rho \mu^{AD}/\rho$, and $\kappa = c_T\mu^{AD}$. The artificial viscosity coefficient, $\mu^{AD}$,
is defined so that it is design-order small in regions where the solution has enough regularity and first-order accurate at strong discontinuities. Further details on how $\mu^{AD}$ is constructed are available in \cite{upperman2023first,upperman2022positivity}. For all test problems presented herein,
the tunable coefficients, $c_\rho$ and $c_T$, are set equal to $0.9$ and $\frac{c_\rho}{\gamma-1}$, respectively.

A necessary condition for selecting a unique, physically relevant solution among possibly many weak solutions of the compressible Navier{--}Stokes equations is the entropy inequality. The original and regularized Navier{--}Stokes equations are both equipped with the same convex scalar entropy function $\mathcal{S}= -\rho s$ and entropy flux $\mathcal{F}=-\rho s \bm{v}$, where $s$ is the thermodynamic entropy. The following inequality can be shown to hold for both the conventional and regularized Navier{--}Stokes equations, assuming the corresponding boundary conditions are entropy-stable (e.g., see~\cite{yamaleev2019entropy}):
\begin{equation}\label{eq:Eineq.3}
  \int_{\hat\Omega} \frac{\partial ( J\mathcal{S})}{\partial t}\mr{d}\hat\Omega
 = \frac{d}{d t} \int_{\hat\Omega} J \mathcal{S}\mr{d}\hat\Omega
  \leq 0.
\end{equation} 
{ Note that all boundary conditions used in the numerical experiments presented in Section \ref{sec:results},  except for the subsonic outflow condition,  are provably entropy-stable.

Since the Brenner and regularized Navier--Stokes equations have nearly identical inviscid and viscous terms, the positivity of the thermodynamic variables in Eq.~(\ref{rns3d}) can be shown in the same  manner as in \cite{feireisl2010new}.  Therefore,}
along with the entropy inequality given by Eq.~(\ref{eq:Eineq.3}), the regularized Navier{--}Stokes equations (Eq.~\eqref{rns3d}) preserve the positivity of thermodynamic variables. 


\section{Semi-discrete 1st-order positivity-preserving entropy{-}stable \\ scheme}
\label{sec:semi-sidcrete}

We start by outlining a baseline semi-discrete first-order positivity-preser-ving entropy{-}stable finite volume (FV) scheme {introduced in \cite{upperman2023first,upperman2021high}} for the regularized Navier{--}Stokes equations (\ref{rns3d}) discretized on high-order hexahedral Legendre-Gauss-Lobatto (LGL) grids.  {A notation similar to that} used in \cite{upperman2023first} is {adopted herein to facilitate} reference to this baseline scheme. The way how {this} baseline scheme can be extended to attain high-order positivity-preserving properties {in space} and unconditional stability in the physical time will be presented in Sections \ref{sec:dual_time-stepping} - \ref{sec:bdf2-positivity}.


\subsection{SBP operators}
\label{sec:sbp}

First, we briefly outline 1D summation-by-parts (SBP) operators used to discretize Eq.~(\ref{rns3d}). The physical domain is divided into $N_{\rm elem}$ non-overlapping discontinuous elements, $[x_1^j, x_{N_p}^j]$, such that $x_1^j=x_{N_p}^{(j-1)}$ for $j=2,\dots,N_{\rm elem}$. The solution of order $p$ in each cell is approximated on $N_p=p+1$ Legendre-Gauss-Lobatto (LGL) points, ${\bf x}^j = \left[x_1^j, \dots, x_{N_p}^j \right]^\top$ (referred to as solution points) for $j=1,\dots,N_{\rm elem}$. This representation provides us with a set of operators, including a quadrature, $\mathcal{P}$, a 1st-derivative differentiation operator, $\mathcal{D}$, and a stiffness matrix, $\mathcal{Q}$. In this work, we restrict ourselves to the diagonal-norm LGL operators only. The main properties of these operators for a fixed order $p$ are as follows.
\begin{enumerate}
  \item For any vector ${\bf x}^l = \left[x_1^l,\dots, x_{N_p}^l\right]^\top$ and powers $l=0, 1,\dots, p$, $\mathcal{D}{\bf x}^l = \mathcal{P}^{-1}\mathcal{Q}{\bf x}^l = l{\bf x}^{l-1}$.
  \item $\mathcal{P}$ is a symmetric positive definite (SPD) matrix.
  \item $\mathcal{Q}+\mathcal{Q}^\top=\mathcal{B}$, where $\mathcal{B} = {\rm diag}(-1,0,\dots,0,1)$.
\end{enumerate}
{A} comprehensive discussion of SBP operators {can be found in}~\cite{fernandez2014review}.

Along with the solution points, we also use an additional set of intermediate points, $\bar{\bf x}^j = \left[\bar{x}_0^j, \dots, \bar{x}_{N_p}^j \right]^\top$ for $j=1,\dots,N_{\rm elem}$. These points, which are referred to as flux points, form a complementary grid whose spacing is equal to the diagonal elements of the positive definite mass matrix $\mathcal{P}$, i.e., 
\begin{equation}\label{eq:fluxpoints}
\bar{x}_{i} - \bar{x}_{i-1} = \mathcal{P}_{ii} \ {\rm for} \ i=1,\dots, N_p. 
\end{equation}
The flux points are instrumental for constructing the first-order positivity-preserving entropy{-}stable scheme defined on high-order LGL elements, which will be discussed in Section \ref{sec:first-order}. As has been proven in \cite{fisher2013discretely}, any 1D SBP discrete differentiation operator $\mathcal{D} = \mathcal{P}^{-1} \mathcal{Q}$ presented above can be recast into the following telescopic flux form:
\begin{equation}
\nonumber
  \mathcal{P}^{-1} \mathcal{Q} {\bf f} 
  = \mathcal{P}^{-1} \Delta \bar{\bf f},
\end{equation}
where $\Delta$ is a $N_p\times (N_p +1)$ matrix corresponding to the two-point backward difference operator, and $\bar{\bf f}$ is a $p$th-order flux vector defined at the flux points~\cite{fisher2013discretely,carpenter2014entropy,fisher2011boundary}. 

For unstructured hexahedral grids, the one-dimensional SBP operators defined on each LGL element can be straightforwardly extended to three spatial dimensions by using tensor product arithmetic (e.g., see~\cite{carpenter2016entropy}). Hereafter, the multidimensional SBP operators defined in the computational domain are denoted with subscripts $\cdot_{d}$, where $d$ is the $d$-th computational coordinate for $d = 1, 2, 3$. Because the scheme is developed { in} three spatial dimensions, with each coordinate defined by the index $d$ and {each element containing} $N_p$ LGL points in {every} direction,  we use the notation $i_d$ to denote the $i$-th LGL point and $\overline{i}_d$ to denote the $i$-th flux point in the $d$-th curvilinear coordinate. To refer to a specific quantity at a point,  we use the notation $\cdot_{i_1,i_2,i_3}$, where $i_1$, $i_2$, and $i_3$ are  indices of the quantity in the first, second, and third coordinates, respectively.


\subsection{Baseline 1st-order positivity-preserving entropy{-}stable scheme}
\label{sec:first-order}

To guarantee the positivity of thermodynamic variables in the presence of strong discontinuities,  we use the first-order positivity-preserving entropy{-}stable finite volume (FV) scheme introduced {in~\cite{upperman2023first}} for the 3D compressible Navier{--}Stokes equations. 
This first-order scheme is constructed in a finite volume manner on the high-order LGL solution points with the flux points acting as control volume edges and can be written in the semi-discrete form as follows:
\begin{equation}\label{eq:first-order} 
  (\hat{{\bf u}}_1)_t
  = \sum\limits_{ d=1}^{3} \left(
  -\mathcal{P}^{-1}_{d} \Delta_{d} \left[ 
  \hat{\bar{{\bf f}}}^{1(I)}_d 
  - \hat{\bar{{\bf f}}}^{1(AD)}_{\hat{\bar{\sigma}},d}
  - \hat{\bar{{\bf f}}}^{1(AD)}_d 
  \right] 
  + \mathcal{D}_{d}\hat{{\bf f}}^{p(V)}_d \right)
  + \sum\limits_{ d=1}^{3} \mathcal{P}^{-1}_{d} 
  \hat{{\bf g}}^1_d, 
\end{equation} 
where {$\hat{{\bf u}}_1= [J]{\bf u}_1$,  $[J]$ is a diagonal matrix composed out of the metric Jacobian computed at the corresponding solution points,}
$\hat{\bar{{\bf f}}}^{1(I)}_d$, $\hat{\bar{{\bf f}}}^{1(AD)}_{\hat{\bar{\sigma}},d}$, and $\hat{\bar{{\bf f}}}^{1(AD)}_d$ are first-order inviscid and artificial dissipation fluxes, $\hat{{\bf f}}^{p(V)}_d$ is a high-order physical viscous flux associated with the $d$-th coordinate, and $\hat{{\bf g}}^1_d$ represents inviscid, viscous, and artificial dissipation penalties \cite{upperman2023first}. 

Note that the discretization of the first-order inviscid fluxes on high-order LGL elements satisfies the geometric conservation law (GCL) equations \cite{upperman2023first,thomas1979geometric}. The first-order inviscid fluxes in Eq.~(\ref{eq:first-order}) are discretized by using an approximation that was first introduced in~\cite{upperman2021high}. These inviscid fluxes are represented as follows: $\hat{\bar{{\bf f}}}^{1(I)}_d = \hat{\bar{{\bf f}}}^{(EC)}_d - \hat{\bar{{\bf f}}}^{(ED)}_d$, where $\hat{\bar{{\bf f}}}^{(EC)}_d $ is a two-point entropy{-}conservative flux and $\hat{\bar{{\bf f}}}^{(ED)}_d$ is an entropy{-}dissipative characteristic flux developed in~\cite{merriam1989entropy}, which is constructed so that it facilitates the pointwise density positivity~\cite{upperman2021high}. The entropy{-}conservative flux, $\hat{\bar{{\bf f}}}^{(EC)}_1$, is defined as follows:
\begin{equation}
\label{1ST_ECFLUX}
  \left\{
  \begin{array}{ll}
  \hat{\bar{{\bf f}}}^{(EC)}_1(\vec{\xi}_{\overline{i}_d})
  = \bar{{ f}}_{S}\left({\bf u}_1(\vec{\xi}_{i_d}),{\bf u}_1(\vec{\xi}_{i_d+1})\right) \hat{\bar{{\bf a}}}^1(\vec{\xi}_{\overline{i}_d}), 
  & \text{for} \, \, 1 \leq \overline{i}_d \leq N_p-1,\\
  \hat{\bar{{\bf f}}}^{(EC)}_1(\vec{\xi}_{\overline{i}_d}) = \bar{{ f}}_{S}\left({\bf u}_1(\vec{\xi}_{\overline{i}_d}),{\bf u}_1(\vec{\xi}_{\overline{i}_d})\right) \hat{\bar{{\bf a}}}^1(\vec{\xi}_{\overline{i}_d}), & \text{for} \, \, \overline{i}_d \in \{0,N_p \}, 
  \end{array}
  \right.
\end{equation}
\begin{equation*}
  \hat{\bar{{\bf a}}}^1(\vec{\xi}_{\overline{i}_d})
  =\sum\limits_{ i=i_d+1}^{N_p} \sum\limits_{ l=1}^{i_d} 2 q_{l,i}
  \frac{\hat{{\bf a}}^1(\vec{\xi}_l) +\hat{{\bf a}}^1(\vec{\xi}_i)}{2}, 
\end{equation*}
where $\hat{{\bf a}}^1$ is the $p$th-order approximation of $J\frac{\partial \xi_1}{\partial {\bf x}}$ satisfying the GCL equations and $\bar{{ f}}_{S}(\cdot,\cdot)$ is any two-point, entropy{-}conservative inviscid flux. The entropy{-}conservative fluxes $\hat{\bar{{\bf f}}}^{(EC)}_2$ and $\hat{\bar{{\bf f}}}^{(EC)}_3$ are defined by similar expressions with the corresponding metric approximation, $\hat{{\bf a}}^2$ and $\hat{{\bf a}}^3$, respectively (for further details, see \cite{upperman2023first}). In the present analysis,  we use the two-point entropy{-}conservative flux developed in~\cite{chandrashekar2013kinetic}. The proofs that the above discretization of the inviscid fluxes provides entropy stability, preserves the positivity of density, and satisfies the GCL equations can be found in \cite{upperman2023first}.

The first-order artificial dissipation fluxes are split into two terms. The $\hat{\bar{{\bf f}}}^{1(AD)}_d$ flux is computed based on the artificial viscosity $\bfnc{\mu}^{AD}$, where the artificial viscosity coefficient is evaluated as the arithmetic average of the corresponding values at the neighboring solution points. The $\hat{\bar{{\bf f}}}^{1(AD)}_{\hat{\bar{\sigma}},d}$ flux is introduced to add mass diffusion guaranteeing density positivity~\cite{upperman2023first}. The physical viscous fluxes $\hat{{\bf f}}^{p(V)}_d $ in Eq.~(\ref{eq:first-order}) are discretized by the same high-order SBP operators used for the high-order entropy{-}stable scheme briefly outlined in Section \ref{sec:high-order-entropy-stable}.


\section{Dual time-stepping method}
\label{sec:dual_time-stepping}

Most, if not all, positivity-preserving entropy{-}stable schemes discretize the time derivative term in the governing equations {by high-order} explicit SSP Runge-Kutta methods. Note, however, that these explicit time integrators impose a CFL-type condition on the time step, which may become stiff for viscous flow simulations. To eliminate this stiffness and preserve the positivity properties,  we propose to discretize the time derivative term in Eq.(\ref{eq:first-order}) by using a dual time-stepping implicit backward difference (BDF) method. 

The main idea of the dual time-stepping method is to introduce an additional derivative term in the pseudotime into the discretized governing equations. As a result, the nonlinear discrete equations at each physical time {step can} be solved by marching to the steady state in the pseudotime. Thus, the dual time-stepping method converts the unsteady problem in the physical time to a steady-state problem in the pseudotime at each physical time step~\cite{merkle1987time}. In the present analysis, the dual time-stepping method {is used} to discretize the physical time derivative by using first- and second-order implicit backward difference (BDF) schemes benefiting from their unconditional stability properties and the pseudotime derivative with an explicit SSP scheme benefiting from its positivity-preserving properties. This approach allows us to derive explicit bounds on the pseudotime step size required for the positivity of thermodynamic {variables}. 
 
 In this section, a new dual time-stepping first-order backward difference (BDF1) scheme that is also first-order accurate in space {is constructed}. Then, this dual time-stepping methodology {is extended} to the corresponding positivity-preserving scheme in Section~{\ref{sec:positivity}} and generalize it to high-order spatial discretizations and the second-order BDF time integrator in Sections~\ref{sec:dual-high} and {\ref{sec:bdf2-positivity}}, respectively.  
 
 The first-order dual time-stepping BDF1 scheme can be written in the following semi-discrete form:
\begin{equation}\label{eq:pseudostep}
  \frac{\partial \hat{\bf u}_1^*}{\partial\tau} = - \frac{\hat{\bf u}_1^* - \hat{\bf u}_1^n}{\Delta t_n} + {\bf R}_1^*,
\end{equation}
where $\hat{\bf u}_1^*$ is a semi-discrete  {solution} in the pseudotime, the ${\bf R}_1^*$ term represents the spatial discretization used, i.e., the right-hand side of Eq.(\ref{eq:first-order}) evaluated at $\hat{\bf u}_1^*$, {$\Delta t_n=t^{n+1}-t^n$} is a physical time step size, and $n$ is {a} current physical time level. When the {pseudotime} derivative converges to zero, i.e., $(\hat{\bf u}_1^*)_{\tau} \rightarrow {\bf 0}$, $\hat{\bf u}_1^*$ converges to ${\hat{\bf u}}_1^{n+1}$ and Eq.~(\eqref{eq:pseudostep}) becomes the standard BDF1 scheme. To obtain an explicit update in the pseudotime and achieve a more favorable bound on the pseudotime step size required for positivity (see Section \ref{sec:positivity}), the $\frac{\partial \hat{\bf u}_1^*}{\partial\tau}$ term in Eq.~(\ref{eq:pseudostep}) is discretized by using the first-order explicit Euler formula and $\hat{\bf u}_1^*$ is set be the solution at the next pseudotime level. Thus, the implicit BDF1 dual time-stepping scheme is given by
\begin{equation}\label{eq:semi-implicit-update}
  \hat{\bf u}_1^{k+1} = \hat{\bf u}_1^k + \Delta\tau_k \left(- \frac{\hat{\bf u}_1^{k+1} - \hat{\bf u}_1^n}{\Delta t_n} + {\bf R}_1^k\right),
\end{equation}
where {$\Delta\tau_k=\tau^{k+1}-\tau^k$} is a grid spacing at the $k$-th pseudotime level,
${\bf R}_1^k$ is the right-hand side of Eq.(\ref{eq:first-order}) evaluated at the current pseudotime level $k$, and $\hat{\bf u}_1^n$ is the discrete solution at the $n$-th physical time level, which remains unchanged during pseudotime iterations. 
\begin{remark}
An alternative approach is to set $\hat{\bf u}_1^*$ equal to the solution at the current pseudotime level, thus leading to
$$
  \hat{\bf u}_1^{k+1} = \hat{\bf u}_1^k + \Delta\tau_k \left(- \frac{\hat{\bf u}_1^{k} - \hat{\bf u}_1^n}{\Delta t_n} + {\bf R}_1^k\right).
$$
Note, however, that a constraint that should be imposed on the pseudotime step to guarantee the density positivity in this case is stricter than that of its counterpart with the explicit Euler integrator in the physical time. Therefore, this scheme is not considered in the present analysis.
\end{remark}

Solving Eq.~(\ref{eq:semi-implicit-update}) for $\hat{\bf u}_1^{k+1}$
yields the following explicit update formula in the pseudotime:
\begin{equation}
 \label{eq:bdf1-dual-time-stepping-update}
  \hat{\bf u}_1^{k+1} 
  = \frac{\Delta t_n}{\Delta t_n + \Delta \tau_k} \hat{\bf u}_1^k 
  + \frac{\Delta \tau_k}{\Delta t_n + \Delta \tau_k} \hat{\bf u}_1^n 
  + \frac{\Delta t_n \Delta \tau_k}{\Delta t_n + \Delta \tau_k} {\bf R}_1^k.
\end{equation}
This update formula is used to converge the solution to the steady-state solution in the pseudotime, 
which can be interpreted as an iterative solver for solving the nonlinear discrete equations at each physical time step.

For the sake of brevity,  we hereafter omit the subscripts $n$ and $k$ in $\Delta t_n$ and $\Delta \tau_k$ and use the following notation:
\begin{equation}
\label{eq:Ctau}
  \hat{\bf u}_1^{k,n} = \hat{\bf u}_1^k + \frac{\Delta \tau}{\Delta t} \hat{\bf u}_1^n,
  \quad C_1^{\tau} = \frac{\Delta t}{\Delta t + \Delta \tau}.
\end{equation}
With this notation, the dual time-stepping BDF1 scheme given by Eq.(\ref{eq:bdf1-dual-time-stepping-update}) becomes
\begin{equation}
 \label{eq:bdf1-update}
 \hat{\bf u}_1^{k+1}
 = C_1^{\tau}\left(\hat{\bf u}_1^{k,n} + \Delta\tau{\bf R}_1^k
 \right),
\end{equation}
where $\hat{\bf u}_{1}^{k+1}=[J]{\bf u}^{k+1}_{1}$, $\hat{\bf u}^{n}_{1}=[J]{\bf u}^{n}_{1}$. 


\section{First-order dual time-stepping positivity-preserving entropy{-}stable scheme}
\label{sec:positivity}

The dual time-stepping BDF1 scheme given by Eq.~\eqref{eq:bdf1-update} very closely resembles the update formular of the explicit Euler scheme. Therefore, the positivity of thermodynamics variables of the dual time-stepping BDF1 scheme can be proven in a similar way to that of the explicit Euler scheme reported in \cite{upperman2023first}. 
These density and temperature positivity proofs are presented next.


\subsection{Positivity of density}
\label{sec:density-positivity}

Let us prove that the dual time-stepping scheme given by Eq.~\eqref{eq:bdf1-update} is positivity-preserving at every pseudotime step if some bounds are imposed on $\Delta \tau$. Consequently, this scheme also preserves the positivity of the thermodynamic variables at each physical time step over the entire time interval of interest. Denoting $\hat{\rho}(x_{i_1i_2i_3})=J_{i_1i_2i_3}\rho(\xi_{i_1i_2i_3})$ and using Eq. (\ref{eq:fluxpoints}), the continuity equation in the 1st-order scheme (Eq.(\ref{eq:first-order})) is given by
\begin{equation}\label{eq:cont}
 (\hat{\rho})_t 
 = - \sum_{d=1}^{3} \frac{\hat{\bar{f}}^{\rho+}_{i_d}
 - \hat{\bar{f}}^{\rho-}_{i_d}}{\Delta \bar{\xi}_{i_d}},
\end{equation}
where $i_d$ is the $i$-th solution node in the $d$-th spatial direction and $\Delta \bar{\xi}_{i_d}=\bar{\xi}_{i_d+1}-\bar{\xi}_{i_d}$ is the distance between the neighboring flux points in the computational domain. The numerical fluxes $\hat{\bar{f}}^{\rho\pm}_{i_d}$ in Eq.(\ref{eq:cont}) are defined as follows:
\begin{equation}\label{eq:flux}
 \hat{\bar{f}}^{\rho\pm}_{i_d}
 = \hat{\bar{m}}_{i_d}^{\pm} - \mathscr{D}_{i_d}^{\pm} \Delta_{i_d}^{\pm} {\rho},
\end{equation}
where $\Delta_{i_d}^+{\rho}={\rho}_{i_d+1}-{\rho}_{i_d}$ and $\Delta_{i_d}^-{\rho}={\rho}_{i_d}-{\rho}_{i_d-1}$,
 $\hat{\bar{m}}_{i_d}^{+}$ and $\hat{\bar{m}}_{i_d}^{-}$ are the momentums associated with the entropy{-}conservative flux given by Eq.~\eqref{1ST_ECFLUX} based on ${\bf u}_1(\xi_{i_d})$, ${\bf u}_1(\xi_{i_d+1})$ and  ${\bf u}_1(\xi_{i_d-1})$, ${\bf u}_1(\xi_{i_d})$,respectively, $\mathscr{D}_{i_d}^{\pm}$ is the corresponding dissipation coefficient whose minimum value is given by $\mathscr{D}_{i_d, \min}^{\pm} = \frac{|\hat{\bar{m}}^{\pm}_{i_d}|}{2{\rho}_{i_d,A}^{\pm}}$, ${\rho}_{i_d,A}^{\pm}=({\rho}_{i_d}+{\rho}_{i_d\pm 1})/2$, and the bar ($\bar{\cdot}$) indicates a quantity defined at a given flux point (for further details, see~\cite{upperman2023first}). 
 
The following theorem presents a sufficient condition that needs to be imposed on the pseudotime step size $\Delta \tau $ in the first-order scheme (Eq. \eqref{eq:first-order}) whose time derivative term is discretized by the dual time-stepping BDF1 method (Eq.\ref{eq:bdf1-update})) to guarantee the pointwise positivity of density. 
\begin{theorem}
\label{COR:posDensMassDiff_1stOrder}
  If the $\hat{\bar{f}}^{\rho\pm}_{i_d}$ flux is defined by Eq.~(\ref{eq:flux}) with $\mathscr{D}^{\pm}_{i_d} \geq \mathscr{D}^{\pm}_{{i_d},\min}= \frac{|\hat{\bar{m}}^{\pm}_{i_d}|}{2{\rho}_{i_d,A}^{\pm}}$, then
 the first-order dual time-stepping scheme given by Eq.~(\ref{eq:bdf1-update}) preserves the positivity of density under the following constraint on $\Delta \tau$: 
  \begin{equation}\label{eq:density-positivity}
    \Delta\tau <
    \min\limits_{i_1, i_2, i_3}\frac{1}
    {\frac 2{J_{i_1i_2i_3}} \sum\limits_{d=1}^{3}
    \frac{\mathscr{D}^{+}_{i_d} + \mathscr{D}^{-}_{i_d}}{\Delta \bar{\xi}_{i_d}}
    - \frac{1}{\Delta t} \frac{\rho^n}{\rho^k}}
    = \Delta\tau^{\rho}_1,
   \end{equation}
 { where ${\rho}^{k}=\rho_{i_1i_2i_3}^{k}$ and ${\rho}^{n}=\rho_{i_1i_2i_3}^{n}$. }
\end{theorem}

\begin{proof}
The conservation of mass equation discretized by the dual time-stepping BDF1 scheme (Eq.(\ref{eq:bdf1-update})) at each solution point is given by
\begin{equation}
 \label{eq:bdf1-density-update}
 \hat{\rho}^{k+1}
 = C_1^{\tau}\left(\hat{\rho}^{k,n} 
 - \Delta\tau \sum\limits_{d=1}^{3}
 \frac{\hat{\bar{f}}^{\rho+}_{i_d}-\hat{\bar{f}}^{\rho-}_{i_d}}{\Delta \bar{\xi}_{i_d}}\right),
\end{equation}
where $\hat{\rho}^{k+1}=J_{i_1i_2i_3}\rho_{i_1i_2i_3}^{k+1}$ and $\hat{\rho}^{n}=J_{i_1i_2i_3}\rho_{i_1i_2i_3}^{n}$. 
The update formula given by Eq.~(\ref{eq:bdf1-density-update}) can be split as
{
\begin{equation}
\nonumber
  \hat{\rho}^{k+1} = C_1^{\tau}
  \left(
  \left[
    \frac{\hat{\rho}^{k,n}}{6} 
    - \Delta\tau\frac{\hat{\bar{f}}^{\rho+}_{i_1}}{\Delta \bar{\xi}_{i_1}}
  \right]
+
   \left[
    \frac{\hat{\rho}^{k,n}}{6} 
    + \Delta\tau\frac{\hat{\bar{f}}^{\rho-}_{i_1}}{\Delta \bar{\xi}_{i_1}}
    \right]
   + \cdots
    \right)
\end{equation}
In the above formula,  only the splitting in the $\xi_1$ direction is shown, while} the same splitting is applied in the $\xi_2$ and $\xi_3$ coordinate directions. 
The $''+''$ and $''-''$ terms can be recast in the following form:
\begin{align}
\nonumber
  \left(
    \frac{\hat{\rho}^{k,n}}{6} 
    - \Delta\tau\frac{\hat{\bar{f}}^{\rho+}_{i_d}}{\Delta \bar{\xi}_{i_d}}
  \right) &= 
  \frac{\hat{\rho}^{k,n}}{6} 
  - \Delta\tau\frac{1}{\Delta \bar{\xi}_{i_d}}
  (\hat{\bar{m}}^+_{i_d} - \mathscr{D}^+_{i_d}\Delta^+_{i_d}\rho),\\
\nonumber
  \left(
    \frac{\hat{\rho}^{k,n}}{6} 
    + \Delta\tau\frac{\hat{\bar{f}}^{\rho-}_{i_d}}{\Delta \bar{\xi}_{i_d}}
  \right) &=
  \frac{\hat{\rho}^{k,n}}{6} 
  + \Delta\tau\frac{1}{\Delta \bar{\xi}_{i_d}}
  (\hat{\bar{m}}^-_{i_d} - \mathscr{D}^-_{i_d}\Delta^-_{i_d}\rho),
\end{align}
for $d=1,2,3$.

Since $\hat{\bar{m}}^+_{i_d}$ and $\hat{\bar{m}}^-_{i_d}$ are scalars, $-\hat{\bar{m}}^+_{i_d}\geq-|\hat{\bar{m}}^+_{i_d}|$ and $\hat{\bar{m}}^-_{i_d}\geq-|\hat{\bar{m}}^-_{i_d}|$. Taking into account that $\mathscr{D}^{\pm}_{i_d} \geq \frac{|\hat{\bar{m}}^{\pm}_{i_d}|}{2{\rho}_{i_d,A}^{\pm}}$ and assuming ${\mathscr{D}^\pm_{i_d}}>0$ yield
\begin{align}
\nonumber
  \begin{split}
    \frac{\hat{\rho}^{k,n}}{6} 
    - \Delta\tau\frac{1}{\Delta \bar{\xi}_{i_d}}
    (\hat{\bar{m}}^+_{i_d} - \mathscr{D}^+_{i_d}\Delta^+_{i_d}\rho)
    &\geq \frac{\hat{\rho}^{k,n}}{6} 
    - \Delta\tau\frac{\mathscr{D}^+_{i_d}}{\Delta \bar{\xi}_{i_d}}
    \left(\frac{|\hat{\bar{m}}^+_{i_d}|}{\mathscr{D}^+_{i_d}} - \Delta^+_{i_d}\rho\right)\\
    &\geq \frac{\hat{\rho}^{k,n}}{6} 
    - \Delta\tau\frac{\mathscr{D}^+_{i_d}}{\Delta \bar{\xi}_{i_d}}
    \left(2\rho^+_{j,A} - \Delta^+_{i_d}\rho\right)\\
    &= \frac{\hat{\rho}^{k,n}}{6} 
    - \Delta\tau\frac{2\mathscr{D}^+_{i_d}}{\Delta \bar{\xi}_{i_d}}
    \rho^k\\
    = \rho^k &\left[
      \frac{J_{i_1i_2i_3}}{6} - \Delta\tau\left(
        \frac{2\mathscr{D}^+_{i_d}}{\Delta \bar{\xi}_{i_d}}
        - \frac{J_{i_1i_2i_3}}{6} \frac{1}{\Delta t} \frac{\rho^n}{\rho^k}
      \right)  
    \right]
  \end{split}
\end{align}
\begin{align}
\nonumber
  \begin{split}
    \frac{\hat{\rho}^{k,n}}{6} 
    + \Delta\tau\frac{1}{\Delta \bar{\xi}_{i_d}}
    (\hat{\bar{m}}^-_{i_d} - \mathscr{D}^-_{i_d}\Delta^-_{i_d}\rho)
    &\geq \frac{\hat{\rho}^{k,n}}{6} 
    - \Delta\tau\frac{\mathscr{D}^-_{i_d}}{\Delta \bar{\xi}_{i_d}}
    \left(\frac{|\hat{\bar{m}}^-_{i_d}|}{\mathscr{D}^-_{i_d}} + \Delta^-_{i_d}\rho\right)\\
    &\geq \frac{\hat{\rho}^{k,n}}{6} 
    - \Delta\tau\frac{\mathscr{D}^-_{i_d}}{\Delta \bar{\xi}_{i_d}}
    \left(2\rho^-_{j,A} + \Delta^-_{i_d}\rho\right)\\
    &= \frac{\hat{\rho}^{k,n}}{6} 
    - \Delta\tau\frac{2\mathscr{D}^-_{i_d}}{\Delta \bar{\xi}_{i_d}}
    \rho^k\\
    = \rho^k &\left[
      \frac{J_{i_1i_2i_3}}{6} - \Delta\tau\left(
        \frac{2\mathscr{D}^-_{i_d}}{\Delta \bar{\xi}_{i_d}}
        - \frac{J_{i_1i_2i_3}}{6} \frac{1}{\Delta t} \frac{\rho^n}{\rho^k}
      \right)  
    \right].
  \end{split}
\end{align}
Summing over all element interfaces,  we have
\begin{equation}
\nonumber
  \hat{\rho}^{k+1} 
  \geq C_1^{\tau} \rho^k
  \left[
    J_{i_1i_2i_3} - \Delta\tau\left(
      2\sum_{l=1}^3\frac{\mathscr{D}^+_{i_d}+\mathscr{D}^-_{i_d}}{\Delta \bar{\xi}_{i_d}}
      - J_{i_1i_2i_3} \frac{1}{\Delta t} \frac{\rho^n}{\rho^k}
    \right)  
  \right]>0.
\end{equation}
If the term in the parentheses in the above equation is less or equal to zero, then no constraint should be imposed on $\Delta\tau$ to guarantee the positivity of $\hat{\rho}^{k+1}$.
If this term in the parentheses is strictly positive and $\Delta \tau$ satisfies 
\begin{equation}
\nonumber
  J_{i_1i_2i_3} - \Delta\tau\left(
    2\sum_{d=1}^3\frac{\mathscr{D}^+_{i_d}+\mathscr{D}^-_{i_d}}{\Delta \bar{\xi}_{i_d}}
    - J_{i_1i_2i_3} \frac{1}{\Delta t} \frac{\rho^n}{\rho^k}
  \right)
  >0,
\end{equation}
then $\hat{\rho}^{k+1} > 0$, which yields the desired result.
\hfill $\square$
\end{proof}


\subsection{Positivity of internal energy}
\label{sec:internal-energy-positivity}

Using the BDF1 dual time-stepping update formula \eqref{eq:bdf1-dual-time-stepping-update} for $\hat{\bm u}_1^{k+1}(\vec{\xi}_{i_1i_2i_3})= J_{i_1i_2i_3} {\bm u}_1^{k+1}$, the internal energy at the $(k+1)$-th pseudotime level can be determined by substituting ${\bm u}_1^{k+1}$ into $(\rho E)^{k+1} = (\rho e)^{k+1} \rho^{k+1} + \frac{(m^{k+1})^2}{2\rho^{k+1}}$ at each solution point, thus leading to the following inequality for $\Delta\tau$ provided that 
all the conditions of Theorem~\ref{COR:posDensMassDiff_1stOrder} are satisfied:
\begin{equation}
 \label{eq:internal-energy-positivity}
 \frac{(\rho e)^{k+1}\rho^{k+1}}{{C_1^{\tau}}^2} 
 = A\left(\frac{\Delta\tau}{J}\right)^2 
 + B\frac{\Delta\tau}{J} 
 + C>0,
\end{equation}
where $(\rho e)^{k+1}$ is the total internal energy of ${\bm u}_1^{k+1}$. The coefficients of the quadratic trinomial can be readily computed by using Eqs.~\eqref{eq:bdf1-dual-time-stepping-update} and are given by
\begin{equation}\label{eq:coef}
  \begin{array}{ll}
    A &= (R_1^{E} R_1^{\rho})^k -\frac{1}{2} 
    \left \| (\bm{R}^m_1)^{k} \right \|^2
    +\frac1{\Delta t}\left({\bm u}_1^{n}\right)^\top \left[ 
      \begin{array}{l}
      \phantom{-} {R}_1^{E}   \\
      -\bm{R}_1^{m}  \\
      \phantom{-} {R}_1^{\rho} \\
      \end{array}
    \right]^k\\
    &\quad+ \frac1{\Delta t^2} \left(
      \rho_1^{n} (\rho E)_1^{n}
      - \left\|{\bm m}_1^n\right\|^2
    \right) \\ \\
    B &= \left({\bm u}_1^{k}\right)^\top\left[ 
      \begin{array}{l}
      \phantom{-} {R}_1^{E}   \\
      -\bm{R}_1^{m}  \\
      \phantom{-} {R}_1^{\rho} \\
      \end{array}
    \right]^k
    + \frac1{\Delta t} \left({\bm u}_1^{k}\right)^\top\left[ 
      \begin{array}{l}
        \phantom{-} {\rho E}_1^{n}   \\
        -{\bm m}_1^{n}  \\
        \phantom{-} \rho_1^{n} \\
      \end{array}
    \right]^n \\ \\
    C &= \rho_1^k (\rho E)_1^k - \frac12 \left\|\bm{m}_1^{k}\right\|^2
    = (\rho e)_1^k\rho_1^k,
  \end{array}
\end{equation}
where ${R}_1^{\rho}$, $\bm{R}_1^{m}$, ${R}_1^{E}$ are the right-hand sides of Eq.~(\ref{eq:first-order}) associated with the continuity, momentum, and energy equations, respectively, and $\left\| \cdot \right\|$ is the Euclidean norm in $\mathbb{R}^3$.
Note that Eq.~(\ref{eq:internal-energy-positivity}) holds for any spatial discretization. 

In \cite{upperman2023first}, it is proven that the first-order FV scheme given by Eq.~\eqref{eq:first-order} with the forward Euler discretization in time preserves the positivity of the internal energy, if the solution at the previous time level is in the admissible set, i.e., ${\rho}^{k}>0$ and $({\rho e})^{k}>0$ for all solution points in the domain. Since this proof of the positivity of the internal energy is solely based on the observation that the vertical intercept of the quadratic trinomial is strictly positive if the solution at the previous time level is in the admissible set, a proof of this statement for the proposed dual time-stepping BDF1 scheme is identical to the one given in~\cite{upperman2023first} and not repeated here.

The high-order discretization used for approximating the physical viscous terms may significantly increase the stiffness of the constraint on $\Delta \tau$ required for the positivity of the internal energy in regions where the solution loses its regularity. Bounding the velocity and temperature gradients in troubled elements by using the conservative and entropy{-}stable limiters developed in~\cite{upperman2023first} eliminates this stiffness of the pseudotime step constraint required for the internal energy positivity.
\begin{remark}
It should be noted that the velocity and temperature limiters and the proof of their entropy stability are independent of the temporal discretization and can be directly used for the proposed implicit dual time-stepping schemes without any modifications.
\end{remark}

The proposed dual time-stepping first-order scheme is entropy-stable in the fully discrete sense if the BDF1 method is used for discretizing the derivative in the physical time. Indeed, all spatial terms of the first-order scheme given by Eq. ~(\ref{eq:first-order}) are entropy{-}dissipative \cite{upperman2023first}. Also, it is well known that the BDF1 scheme is unconditionally entropy{-}stable~\cite{tadmor2003entropy}. Therefore, if the physical time derivative is discretized by the BDF1 method, it can be readily shown that the entropy inequality holds in the fully discrete sense.  {This follows from the fact that} both the BDF1 discretization of the time derivative term and all discretized spatial terms in Eq. (\ref{eq:first-order}) are entropy{-}dissipative assuming that the residual of the pseudotime iterations is driven {below some} suitable norm of the entropy dissipation generated by the 1st-order BDF1 scheme.

\section{Dual time-stepping positivity-preserving entropy{-}stable BDF1 scheme with high-order spatial discretization}
\label{sec:dual-high}
{In this section,} the first-order dual time-stepping positivity-preserving entropy{-}stable scheme presented in the foregoing section {is extended} to high-order spatial discretizations. The high-order scheme is constructed by combining the first-order dual time-stepping positivity-preserving scheme given by Eqs.(\ref{eq:bdf1-dual-time-stepping-update},\ref{eq:first-order}) with its high-order positivity-violating counterpart such that the resultant scheme is positivity{-}preserving, entropy{-}stable, and high-order accurate in regions where the solution is sufficiently smooth. 

\subsection{High-order positivity-violating entropy{-}stable scheme}
\label{sec:high-order-entropy-stable}
In {this} section,  we briefly outline how the positivity-violating entropy{-}stable scheme is constructed.
Replacing the low-order spatial FV operators in Eq.(\ref{eq:first-order}) with the high-order spectral collocation operators defined on the same $p$th-order LGL solution points (see Section~\ref{sec:sbp}), the corresponding semi-discrete high-order scheme for the regularized Navier{--}Stokes equations (Eq.(\ref{rns3d})) is given by 
\begin{equation}\label{eq:semi-discrete-ssdc}
  \left(\hat{{\bf u}}_p\right)_t = 
  - \sum\limits_{ d=1}^{3} 
  \mathcal{P}^{-1}_{d} \Delta_{d}\hat{\bar{{\bf f}}}^{p(I)}_d 
  + \mathcal{D}_{d} \left[
    \hat{{\bf f}}^{p(V)}_d 
    + \hat{{\bf f}}^{p(AD)}_d 
  \right]
  + \sum\limits_{ d=1}^{3} \mathcal{P}^{-1}_{d} 
    \hat{{\bf g}}^p_d,
\end{equation} 
where $\hat{{\bf u}}_p = [J] {\bf u}_p$, $\hat{\bar{{\bf f}}}^{p(I)}_d$ for $d=1, 2, 3$ {is a} $p$th - order contravariant inviscid entropy{-}conservative {flux} defined at the flux points, and $\hat{{\bf g}}^p_d$ {includes the} boundary, interface and artificial dissipation penalty terms.  For the full definitions and extended discussion of these fluxes and penalty terms, we refer the reader  to~\cite{carpenter2016entropy}. 

The high-order contravariant viscous and artificial dissipation fluxes in Eq.(\ref{eq:semi-discrete-ssdc}) are constructed as follows:
\begin{equation}
\nonumber
  \hat{{\bf f}}^{p(vis)}_l = \sum\limits_{ d=1}^{3} [\hat{a}^l_d] {\bf f}^{p(vis)}_{d} 
  , \quad 
  {\bf f}^{p(vis)}_{d} = \sum\limits_{ i=1}^{3} [c^{(vis)}_{d,i}] {\bf\Theta}_{d},
\end{equation}
where $\hat{{\bf f}}^{p(vis)}_l = \hat{{\bf f}}^{p(V)}_l$ or $\hat{{\bf f}}^{p(vis)}_l = \hat{{\bf f}}^{p(AD)}_l$, $[c^{(vis)}_{d,i}], 1 \leq d, i \leq 3$ are the corresponding block-diagonal viscosity {matrices} with $5\times 5$ blocks, and $[\hat{ a}^l_{d}]$ is the $p$th-order approximation of $J\frac{\partial \xi_l}{\partial x_d}$ satisfying the geometric conservation law (GCL) equations \cite{thomas1979geometric}. 
The matrices $[c^{(vis)}_{i,l}],$ $1 \leq i, l \leq 3$ satisfy the following properties $[\left( c^{(vis)}_{i,l} \right)^T] = [c^{(vis)}_{l,i}]$ and $\sum\limits_{ i=1}^{3}\sum\limits_{ l=1}^{3} {\bf v}^T[c^{(vis)}_{i,l}] {\bf v} \geq 0, \forall {\bf v} \in \mathbb{R}^5$, i.e., the full viscous/artificial dissipation tensor is symmetric positive semi-definite. The gradient of the entropy variables, ${\bf\Theta}_{d}$, is discretized by using an approach developed in \cite{carpenter2014entropy}, which closely resembles the local discontinuous Galerkin (LDG) method developed in~\cite{cockburn1998local}.

The fully discrete dual-time-stepping variant of the above scheme with the implicit BDF1 discretization in the physical time is given by
\begin{equation}
 \label{eq:bdf1-update-high}
 \hat{\bf u}_p^{k+1}
 = C_1^{\tau}\left(\hat{\bf u}_p^{k,n} + \Delta\tau{\bf R}_p^k
 \right),
\end{equation}
 where $\hat{\bf u}_p^{k,n} = \hat{\bf u}_p^k + \frac{\Delta \tau}{\Delta t} \hat{\bf u}_p^n$, $C_1^{\tau}$ is given by Eq.(\ref{eq:Ctau}), and ${\bf R}_p^k$ is the right-hand side of Eq.~(\ref{eq:semi-discrete-ssdc}) evaluated at the previous pseudotime level $\hat{\bf u}_p^k$.
This fully discrete  {spectral} collocation scheme is conservative and stable in the entropy sense. The conservation follows immediately from the telescopic flux form of the inviscid terms and the SBP form of the viscous and artificial dissipation terms. The entropy stability of the spatial Navier{--}Stokes terms in Eq.~(\ref{eq:semi-discrete-ssdc}) is proven in~\cite{carpenter2014entropy}, and the entropy dissipation properties of the artificial dissipation terms are shown in~\cite{upperman2022positivity, upperman2023first}. The entropy stability of the BDF1 time integrator is proven in \cite{tadmor2003entropy}. 
It should be noted, however, that this dual time-stepping entropy{-}stable BDF1 scheme with the high-order SBP discretization in space does not in general preserve the positivity of thermodynamic variables, because the high-order dissipation operators do not satisfy the maximum principle.


\subsection{High-order positivity{-}preserving entropy{-}stable flux-limiting \newline scheme}
\label{sec:high-order} 

To construct a new {BDF1} dual time-stepping positivity-preserving entropy{-}stable scheme {with high-order spatial discretization}, the {spatial} first-order positivity-preserving (Eq.(\ref{eq:bdf1-update})) and high-order positivity-violating (Eq.(\ref{eq:bdf1-update-high})) {operators} are combined on each LGL element using the flux-limiting technique developed in~\cite{yamaleev2023high} as follows:
\begin{align}
  \begin{split}\label{eq:limiting-solution}
    \hat{\bf u}^{k+1}(\theta_f) 
    &= C_1^{\tau}\left(\hat{\bf u}^{k,n} 
    + \Delta\tau \left[
      \theta_f {\bf R}^k_p
      + (1-\theta_f) {\bf R}^k_1
    \right] \right) 
  \end{split}
\end{align}
where the flux limiter $\theta^k_f $ $(0 \le \theta^k_f \le 1)$ is a constant on each element \cite{yamaleev2023high} and $\hat{\bf u}^{k+1}_p = \hat{\bf u}^{k+1}|_{\theta_f=1}$ and $\hat{\bf u}^{k+1}_1 = \hat{\bf u}^{k+1}|_{\theta_f=0}$ are $p$th- and 1st-order numerical solutions, respectively, which are defined on the same Legendre-Gauss-Lobatto (LGL) elements with {the} high-order metric terms.

Since $0<C_1^{\tau}<1$, a proof that the high-order scheme given by Eq.~(\ref{eq:limiting-solution}) guarantees pointwise positivity of density and temperature is nearly identical to that presented in~\cite{upperman2021high,yamaleev2023high} for the same flux-limiting entropy{-}stable scheme with the explicit Euler discretization in the physical time. The entropy stability of the flux-limiting scheme (Eq.~(\ref{eq:limiting-solution})) follows immediately from the fact that this hybrid scheme is a linear convex combination of two entropy{-}stable schemes on each high-order element. Further details of the positivity and entropy stability of the high-order flux-limiting scheme can be found in~\cite{upperman2021high,yamaleev2023high}.


\section{Dual time-stepping BDF2 positivity-preserving scheme}
\label{sec:bdf2-positivity}

Let us now show how the dual time-stepping BDF1 positivity-preserving entropy{-}stable scheme with the high-order spatial discretization can be extended to the implicit BDF2 time integrator. 
Similar to the dual time-stepping BDF1 scheme presented in Section~\ref{sec:dual_time-stepping}, its BDF2 counterpart can be derived by adding the first-order explicit discretization of the derivative in the pseudotime to the conventional implicit BDF2 scheme in the physical time, thus leading to 
\begin{equation}
\nonumber
 \frac{\hat{\bf u}^{k+1} - \hat{\bf u}^k}{\Delta \tau}
 = - \frac{3\hat{\bf u}^{k+1} - 4\hat{\bf u}^n + \hat{\bf u}^{n-1}}{2 \Delta t} + {\bf R}^k,
\end{equation}
$$
{\bf R}^k = \theta_f {\bf R}^k_p + (1-\theta_f) {\bf R}^k_1,
$$
where $\hat{\bf u}^{k}$ and $\hat{\bf u}^{n}$ are solution vectors at the pseudotime level $k$ and physical time level $n$, 
${\bf R}_1^k$ and ${\bf R}_p^k$ are the right-hand sides of Eqs.~(\ref{eq:first-order}) and (\ref{eq:semi-discrete-ssdc}) evaluated at the previous pseudotime level $\hat{\bf u}^k$, respectively, and $\theta_f$ is the flux limiter presented in Section~\ref{sec:high-order}.

Thus, the update formula for the solution vector $\hat{\bf u}^{k+1}$ at the pseudotime level $k+1$ is given by
\begin{equation}
\label{dts_BDF2}
  \hat{\bf u}^{k+1}
  =  
  \frac{2\Delta t}{2\Delta t+3\Delta \tau}      \hat{\bf u}^k
  + \frac{4\Delta \tau}{2\Delta t+3\Delta \tau}     \hat{\bf u}^n
  - \frac{\Delta \tau}{2\Delta t+3\Delta \tau}     \hat{\bf u}^{n-1}
  + \frac{2\Delta t\Delta \tau}{2\Delta t+3\Delta \tau} {\bf R}^k.
\end{equation}
Introducing the following quantities
\begin{equation}
 \label{eq:bdf2-combination}
  \hat{\bf u}^{k,n}
  = 
   \hat{\bf u}^k
  + \frac{2\Delta \tau}{\Delta t}     \hat{\bf u}^n
  - \frac{\Delta \tau}{2\Delta t}     \hat{\bf u}^{n-1}
\end{equation}
and
\begin{equation}
\nonumber
  C_2^\tau = \frac{2\Delta t}{2\Delta t + 3\Delta \tau},
\end{equation}
the dual time-stepping {BDF2} update formula given by Eq.(\ref{dts_BDF2}) can be recast in the following form:
\begin{equation}
 \label{eq:bdf2-update}
 \hat{\bf u}^{k+1}
 = C_2^{\tau}\left(\hat{\bf u}^{k,n} + \Delta\tau{\bf R}^k
 \right).
\end{equation}

The above dual time-stepping BDF2 scheme is nearly identical to its BDF1 counterpart given by Eq. (\ref{eq:bdf1-update}). As a result, all proofs of the positivity of density and internal energy presented in Sections ~\ref{sec:positivity} can be straightforwardly extended to this dual time-stepping BDF2 scheme.


\subsection{Density positivity}
\label{sec:bdf2-density}

The positivity of density for the dual time-stepping BDF2 scheme given by Eq.(\ref{eq:bdf2-update}) can be proven following the same arguments as in Theorem~\ref{COR:posDensMassDiff_1stOrder}, thus leading to the following constraint on $\Delta\tau$:
\begin{equation}
\nonumber
  \Delta\tau < 
  \min\limits_{i_1, i_2, i_3}\frac{1}
  {\frac 2{J_{i_1i_2i_3}} \sum\limits_{d=1}^{3}
  \frac{\mathscr{D}^{+}_d + \mathscr{D}^{-}_d}{\Delta \bar{\xi}_d}
  - \left(
    \frac{2}{\Delta t} \frac{\rho^{n}}{\rho^k}
    - \frac{1}{2\Delta t} \frac{\rho^{n-1}}{\rho^k}
  \right)}
  = \Delta{\tau^{\rho}_2} ,
\end{equation}
provided that $\mathscr{D}^{\pm}_{i_d} \geq \mathscr{D}^{\pm}_{{i_d},\min}= \frac{|\hat{\bar{m}}^{\pm}_{i_d}|}{2{\rho}_{i_d,A}^{\pm}}$.

\subsection{Internal energy positivity}
\label{sec:bdf2-internal-energy}

We now look into {the} positivity of the internal energy of the dual time{-}stepping BDF2 scheme.  Substituting 
$\hat{\bm u}^{k+1}(\vec{\xi}_{i_1i_2i_3})= J_{i_1i_2i_3} {\bm u}^{k+1}$ given by Eq.~(\ref{eq:bdf2-update}) into the $(\rho E)^{k+1} = (\rho e)^{k+1} \rho^{k+1} + \frac{(m^{k+1})^2}{2\rho^{k+1}}$ at each solution point yields the following quadratic equation for $\Delta \tau$:
\begin{equation}
 \label{eq:internal-energy-BDF2}
 \frac{(\rho e)^{k+1}\rho^{k+1}}{{C_2^{\tau}}^2} 
 = A\left(\frac{\Delta\tau}{J}\right)^2 
 + B\frac{\Delta\tau}{J} 
 + C>0,
\end{equation}
where
$$
  \begin{array}{ll}
    A &= ({R}^{E} {R}^{\rho})^k - \frac{1}{2} \left \| ({\bm R}^{m})^k \right \|^2
    +\left(\frac{2}{\Delta t}{{\bm u}}^{n}-\frac{1}{2\Delta t}{{\bm u}}^{n-1}\right)^\top \left[ 
      \begin{array}{l}
      \phantom{-} {R}^{E}  \\
      -{\bm R}^{m}  \\
      \phantom{-} {R}^{\rho} \\
      \end{array}
    \right]^k \\
    &+ \left(\frac{2}{\Delta t}\rho^{n}-\frac{1}{2\Delta t}\rho^{n-1}\right)
    \left(\frac{2}{\Delta t}(\rho E)^{n}-\frac{1}{2\Delta t}(\rho E)^{n-1}\right)
    - \left\|\frac{2}{\Delta t}{\bm m}^{n}-\frac{1}{2\Delta t}{\bm m}^{n-1}\right\|^2, \\ \\
    B &= \left({{\bm u}}^{k}\right)^\top\left[ 
      \begin{array}{l}
      \phantom{-} {R}^{E}  \\
      -{\bm R}^{m}  \\
      \phantom{-} {R}^{\rho} \\
      \end{array}
    \right]^k
    + \left(\frac{2}{\Delta t}{{\bm u}}^{n}-\frac{1}{2\Delta t}{{\bm u}}^{n-1}\right)^\top\left[ 
      \begin{array}{l}
      \phantom{-} \rho E   \\
      -{{\bm m}}  \\
      \phantom{-} \rho \\
      \end{array}
    \right]^k, \\ \\
    C &= (\rho e)^k\rho^k.
  \end{array}
$$
In the above equation,  we use similar notations as in Eq.~(\ref{eq:coef}).
Thus, if the solution ${\bm u}^k$ at the previous pseudotime step is in the admissible set, then the quadratic trinomial has the positive vertical intercept $C$, which implies that there always exists $\Delta \tau:$ $\Delta{\tau^{\rho}_2} \ge \Delta \tau > 0$ such that the inequality Eq.~(\ref{eq:internal-energy-BDF2}) holds and the proposed dual time-stepping BDF2 scheme {preserves} the positivity of both internal energy and density.

\begin{remark}
The proposed methodology can be directly applied to construct dual time-stepping multistep positivity-preserving schemes beyond the BDF1 and BDF2 time integrators. Note, however, that the third- and higher-order BDF schemes are not $A$-stable and therefore not considered herein.
\end{remark}

\section{Numerical Results}
\label{sec:results}

We now demonstrate the design order of accuracy, positivity-preserving properties, and the convergence performance of the proposed dual time-stepping schemes for both steady-state and unsteady benchmark viscous flows. 
{ For all numerical experiments presented herein, the positivity of density and internal energy are enforced for all collocation points at each peudotime step as described in Setions \ref{sec:positivity}--\ref{sec:bdf2-positivity}.
The} flow quantities are non-dimensionalized as follows: $t=\frac{t'}{L^*/v^*}$, $x_d=\frac{x_d'}{L^*}$, $\rho=\frac{\rho'}{\rho^*}$, $v_i=\frac{v_i'}{v^*}$, $\mu=\frac{\mu'}{\mu^*}$, $p=\frac{p'}{p^*}$, $R=\frac{R'}{R^*}$, $E=\frac{E'}{\gamma c_{\rm v}^* T^*}$, $\kappa=\frac{\kappa'}{\kappa^*}$ and $c_{\rm v}=\frac{c_{\rm v}'}{c_{\rm v}^*}$. The dimensional quantities are indicated with a prime, and the reference quantities are indicated with an asterisk. Thus, the non-dimensional pressure, total and kinetic energy variables are given by
\begin{align}
\nonumber
  &p=\rho T, &E=\frac{1}{\gamma} T + E_k, &&E_k=\frac{(\gamma-1)M_{\infty}^2}{2}u^2.
\end{align}

The $L_1$, $L_2$, and $L_{\infty}$ error norms of the solution vector are computed as follows:
\begin{align}
\nonumber
  \begin{split}
    \|{\bf u} - {\bf u}^\text{ex}\|_{L_1}
    &= \frac{\sum_{j=1}^{N_{\rm elem}} |{\bm u}_j - {\bm u}_j^\text{ex}|^\top \mathcal{P} J_j 1_5}
    {5 \sum_{j=1}^{N_{\rm elem}} 1_5^\top \mathcal{P} J_j 1_5},
  \end{split}\\
\nonumber
  \begin{split}
    \|{\bf u} - {\bf u}^\text{ex}\|_{L_2}
    &= \sqrt{\frac{\sum_{j=1}^{N_{\rm elem}} ({\bm u}_j - {\bm u}_j^\text{ex})^\top \mathcal{P} J_j ({\bm u}_j - {\bm u}_j^\text{ex})}
    {5 \sum_{j=1}^{N_{\rm elem}} 1_5^\top \mathcal{P} J_j 1_5}},
  \end{split}\\
\nonumber
  \begin{split}
    \|{\bf u} - {\bf u}^\text{ex}\|_{L_\infty}
    &= \max_{1\leq j\leq N_{\rm elem}} 
    \max_{1\leq i \leq N_\text{nodes}} \max\limits_{1\leq l \leq 5}
    |{u}^l_j(\vec{\xi}_i) - \left({u}^{\text{ex}}\right)_j^l(\vec{\xi}_i)|.
  \end{split}
\end{align}

The relative and absolute errors are defined as follows
\begin{align}
\nonumber
  &\epsilon_{abs}
  = \frac{\|{\bf u}^{k+1}-{\bf u}^k\|_{L^2}}{\Delta\tau},
  &&\epsilon_{rel}
  = \frac{\|{\bf u}^{m+1}-{\bf u}^k\|_{L^2}}
  {\|{\bf u}^{1}-{\bf u}^0\|_{L^2}}
  \frac{\Delta\tau_0}{\Delta\tau}.
\end{align}


\subsection{3D viscous shock problem}
\label{sec:convergence}

\begin{table}
  \centering
  \caption{Convergence rates of the BDF2 scheme with $p=4$ on globally refined nonuniform grids. }
  \label{tab:vs-temporal-Re10-order-3D-pert}
  \begin{tabular}{cc|cccc}
    $\Delta t$ & $N_{elem}$ & $L^1$ & rate & $L^2$ & rate \\
    \hline
    $3.2\times 10^{-1}$ & $3^3$ & $3.67\times 10^{-4}$ & --   & $6.32\times 10^{-4}$ & --   \\
    $1.6\times 10^{-1}$ & $6^3$ & $1.07\times 10^{-4}$ & $1.77$ & $2.82\times 10^{-4}$ & $1.16$ \\
    $8.0\times 10^{-2}$ & $12^3$ & $1.47\times 10^{-5}$ & $2.87$ & $3.36\times 10^{-5}$ & $3.07$ \\
    $4.0\times 10^{-2}$ & $24^3$ & $3.44\times 10^{-6}$ & $2.10$ & $7.39\times 10^{-6}$ & $2.19$ \\
    \hline
  \end{tabular}
\end{table}

First,  we verify the design order of accuracy of the proposed dual time-stepping BDF2 spectral collocation scheme by solving the 3D viscous shock flow on a sequence of randomly perturbed nonuniform hexahedral grids. The freestream flow parameters for this test case are set as follows: $M_{\infty}=2.5$, $Re_{\infty}=10$, and the flow direction at $[1\ \ 1\ \ 1]^\top$. Table~\ref{tab:vs-temporal-Re10-order-3D-pert} shows the results obtained with the BDF2 spectral collocation scheme with the basis polynomial functions of degree $p=4$ computed on nonuniform grids with $N_\text{elem}=3^3, 6^3, 12^3, 24^3$ elements. For this sequence of grids in the space-time domain, the spatial error is dominated by the second-order temporal error. As one can see from Table~\ref{tab:vs-temporal-Re10-order-3D-pert}, the dual time-stepping BDF2 positivity-preserving entropy{-}stable scheme demonstrates the second-order convergence rate on the finest grid.


\subsection{2D Shock/boundary layer interaction}

The next test problem is the oblique shock boundary layer interaction (SBLI). For this steady-state SBLI problem, the derivative in the physical time is discretized by the implicit BDF1 scheme, while the explicit Euler scheme is used to advance the solution in the pseudotime. We consider two different Mach numbers, $M_{\infty}=2.15$ and $6.85$, and impinging angles $\alpha = 30.8^o$ and $11.8^o$, respectively. The Reynolds and Prandtl numbers are set {to} $Re_{\infty}=10^5$ and $Pr=0.72$ for both cases. The entropy{-}stable no-slip boundary conditions developed in~\cite{dalcin2019conservative} are imposed on a flat plate for $0\le x \le 2$, while the Euler wall boundary conditions are used in the upstream of the plate leading edge $-0.2 \le x <0$. The supersonic inflow boundary conditions are imposed at the left and top boundaries. On the right boundary, the subsonic outflow boundary conditions are used in the boundary layer region, whereas the supersonic outflow is imposed on the rest of the boundary.

For the SBLI test problem, the physical time step used in the dual time-stepping BDF1 scheme is initially fixed and ramped down after a user-specified physical time has passed. As for the pseudotime, the step size is selected to satisfy both the positivity and CFL conditions. Here,  we use the following notation to denote the upper pseudotime step bounds required for {the} positivity of density and temperature: $\Delta \tau^{\rho}_{FE}${, } $\Delta {\tau^{\rho}_1}$, and $\Delta {\tau^{\rm IE}_1}$, {where} $\Delta \tau^{\rho}_{FE}$ is the time step size required for density positivity of the explicit forward Euler scheme presented in \cite{upperman2023first}.

For all SBLI test cases considered, an initial guess is obtained by using the forward Euler positivity-preserving entropy{-}stable $p=2$ spectral collocation scheme. Then, the solution{s are} advanced in the physical time using the explicit forward Euler and implicit BDF1 {dual time-stepping} positivity-preserving entropy{-}stable $p=4$ schemes.


\subsubsection{${M_{\infty}=2.15}$ case}


\begin{table}
  \centering
  \caption{Parameters of the BDF1 scheme for the SBLI $M_{\infty}=2.15$ case. The $\Delta t$ columns show the rampdown start and end values, delayed by $4.5$ time units. }
  \begin{tabular}{l|ccccccc}
  \label{tab:osbli-Ma2-cases}
    Alias & solver & $\Delta t_\text{init}$ & $\Delta t_\text{fin}$ & $\Delta\tau$ & $\|{\bf u}^{m+1}-{\bf u}^m\|_{L^2}$\\
    \hline
    BDF1-162  & pseudo BDF1 & $10^{-1}$ & $5\times 10^{-6}$ & $2\Delta \tau^{\rho}_{FE}$ & $10^{-8}$\\
    BDF1-152  & pseudo BDF1 & $10^{-1}$ & $5\times 10^{-5}$ & $2\Delta \tau^{\rho}_{FE}$ & $10^{-8}$\\
    BDF1-153 & pseudo BDF1 & $10^{-1}$ & $5\times 10^{-5}$ & $3\Delta \tau^{\rho}_{FE}$ & $10^{-8}$\\
  \end{tabular}
\end{table}


\begin{figure}
  \centering
  \begin{subfigure}{0.49\textwidth}
    \centering
    \includegraphics[height=0.245\textheight]{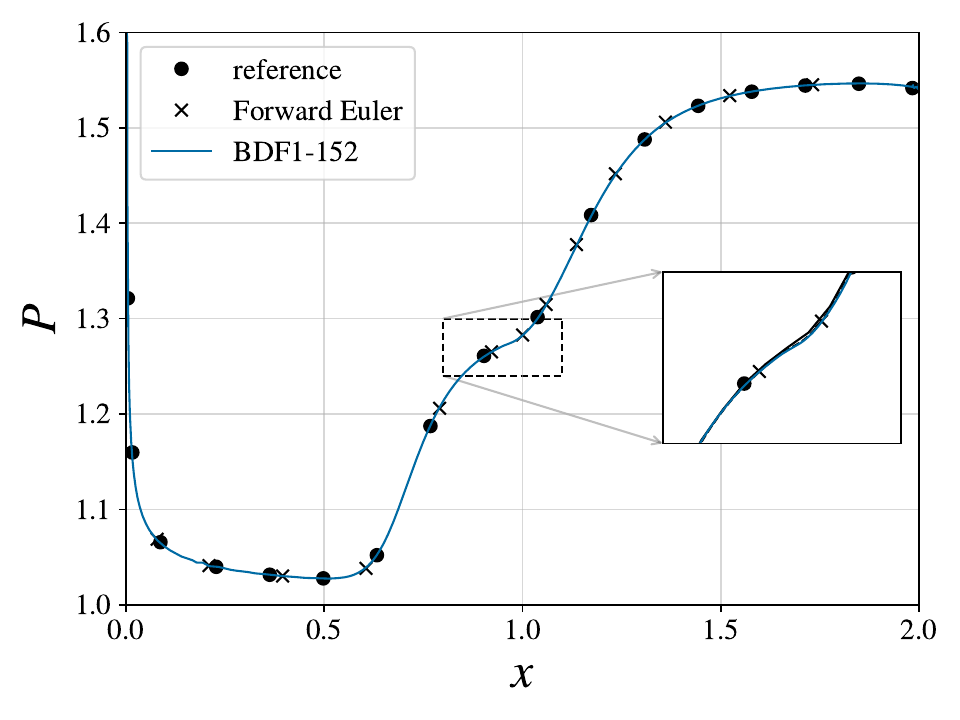}
  \end{subfigure}
  \begin{subfigure}{0.49\textwidth}
    \centering
    \includegraphics[height=0.245\textheight]{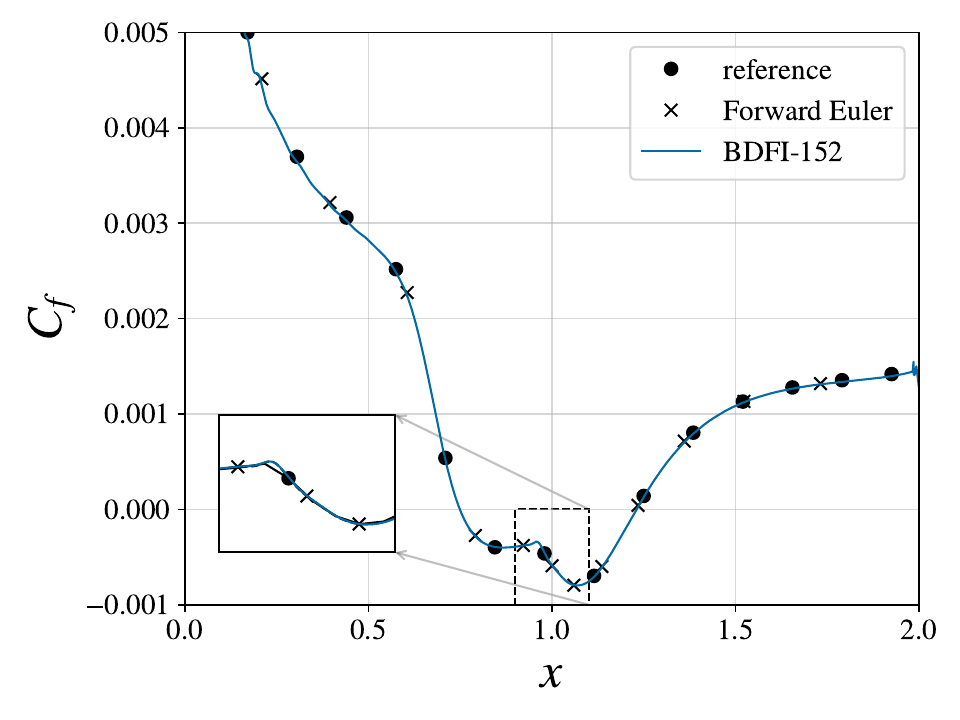}
  \end{subfigure}
  \caption{The wall pressure (left panel) and skin friction coefficient computed with the BDF1 dual time-stepping and forward Euler $p=4$ schemes on the $N_\text{elem}=12,012$ grid and the DG $p=6$ method on the $N_\text{elem}=11,041$ grid \cite{blanchard2016sbli} for the $M_{\infty}=2.15$ SBLI case.}
  \label{fig:osbli-12k--Ma215-p4-wall}
\end{figure}

For this test case, the height of the domain is set equal to $1$, and the grid consists of $N_{elem}=12,012$ LGL $p=4$ elements. For the dual time-stepping BDF1 scheme,  we use three sets of parameters, which are shown in Table \ref{tab:osbli-Ma2-cases},  where the main difference lies in the upper bound on $\Delta \tau$ used in the computations. The corresponding case names are encoded {such that} the first number {denotes} the exponent of the {initial} time step size, the second number {indicates} the ramp down limit exponent of the time step size, and the last number {is} discussed next. Since the positivity-preserving bound given by Eq.~(\ref{eq:density-positivity}) allows for a larger pseudotime step size than that of the explicit forward Euler scheme \cite{upperman2023first}, the pseudotime step is selected such that it exceeds $\Delta \tau^{\rho}_{FE}$ by a factor {of} $2$ or $3$, which represents the third number in the case name. Note, however, that the positivity of thermodynamic variables is checked at each pseudotime step and the density and temperature positivity bounds are always enforced on $\Delta\tau$, if at any solution point either $\rho^{k+1} < \epsilon^{\rho}$ or $(\rho e)^{k+1}<\epsilon^{{\rm IE}}$, where $\epsilon^{\rho}$ and $\epsilon^{{\rm IE}}$ are user-defined small positive constants.

The wall pressure and {the} skin friction coefficient computed with the dual time-stepping BDF1 and explicit forward Euler (FE) schemes with the $p=4$ spectral collocation discretization in space and the reference solution obtained using the $p = 6$ discontinuous Galerkin method \cite{blanchard2016sbli} for the $M_{\infty} = 2.15$ SBLI problem are presented in Fig.~\ref{fig:osbli-12k--Ma215-p4-wall}. Since for this steady state problem, the results obtained with the proposed dual time-stepping BDF1 scheme are nearly identical to each other for all values of the physical and pseudotime step sizes used, only the BDF1-152 results are presented herein.
As one can see from this comparison, the BDF1 and FE density and pressure profiles are nearly
indistinguishable from each other and demonstrate excellent agreement with the reference
solution, thus indicating that the proposed dual time-stepping positivity-preserving entropy{-}stable scheme converges to the same solution obtained with the FE scheme and does not over-dissipate the SBLI solution.


\begin{figure}
    \centering
    \includegraphics[width=0.75\textwidth]{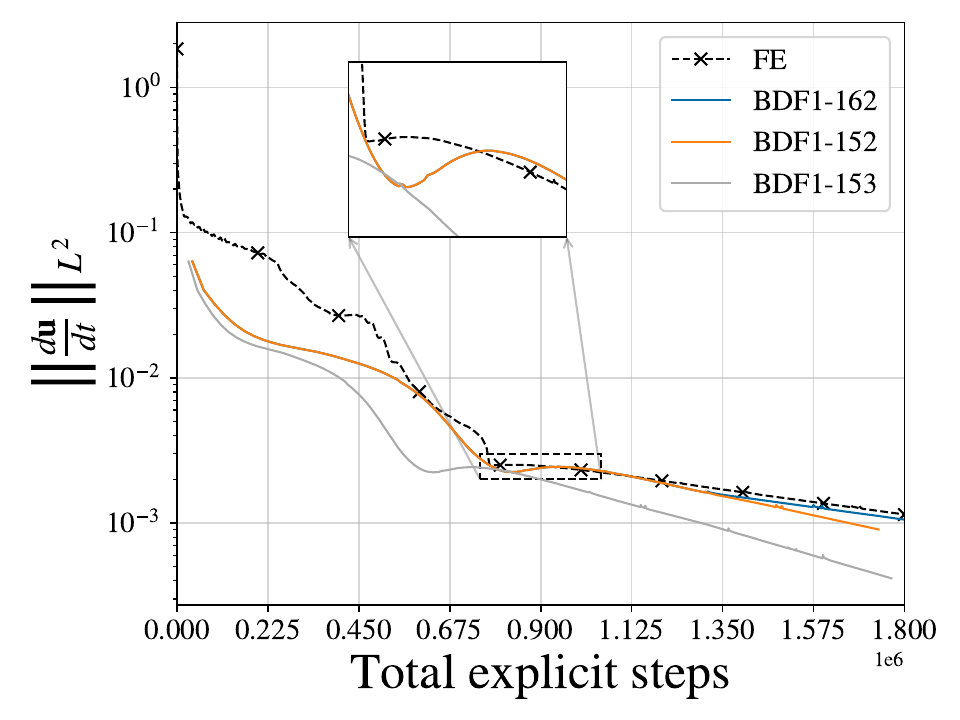}
    \caption{Convergence histories obtained with the forward Euler and dual time-stepping BDF1 schemes for the SBLI $M_{\infty}=2.15$ case.}
    \label{fig:Ma2-total-steps}
\end{figure}

Figure \ref{fig:Ma2-total-steps} shows convergence histories of the forward Euler and dual-time-stepping BDF1 schemes for the SBLI problem at $M_{\infty}=2.15$. As one can see in the figure, lowering the time step of the BDF1 scheme close to the level of $\Delta\tau^{\rho}_{\rm FE}$ (case BDF1-162) eventually gives the same convergence slope obtained with the explicit forward Euler scheme, thus indicating that no further benefit is observed even though its density positivity bound (value of $\Delta {\tau^{\rho}_{\rm 1}}$) is orders of magnitude higher than $\Delta\tau^{\rho}_{\rm FE}$. Instead, the convergence rate increases if the BDF1 time step $\Delta t$ is set about one order of magnitude higher than that of the baseline forward Euler scheme (case BDF1-152). This effect is due to the second term in the denominator of the bound given by Eq.~\eqref{eq:density-positivity}. Further acceleration in {convergence can be achieved by allowing} the pseudotime step size to {exceed} $\Delta{\tau^{\rho}_{\rm 1}}$ { without violating the positivity of density}. As follows from Fig.~\ref{fig:Ma2-total-steps}, the convergence can be accelerated by nearly a factor of 2 when $\Delta\tau = 3 \Delta\tau^{\rho}_{\rm FE}$. It should be noted, however, that increasing the physical time step size much higher than $10 \Delta\tau^{\rho}_{\rm FE}$ significantly slows down the convergence. The main reason for this behavior is the fact that the solution at the previous physical time level becomes a poor initial guess for the forward Euler integrator if $\Delta t$ is very large, thus drastically increasing the number of pseudotime iterations required for convergence at each physical time step. Note that for this test problem, the temperature positivity constraint is several orders of magnitude higher than $\Delta{\tau^{\rho}_{\rm 1}}$ and imposes no additional constraint on $\Delta\tau$.


\subsubsection{${M_{\infty}=6.85}$ case}

\begin{figure}
  \centering
  \includegraphics[width=0.95\textwidth]{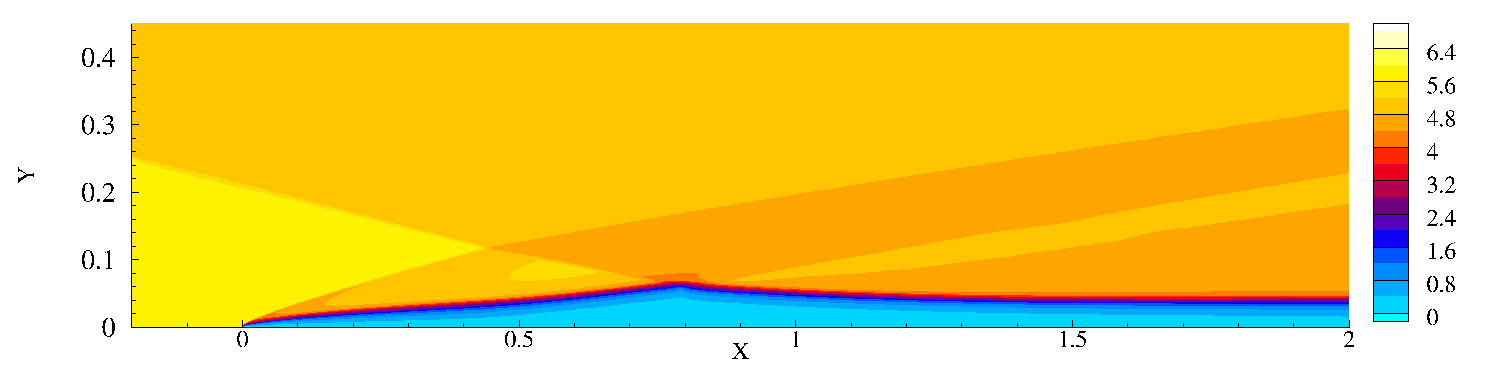}
  \caption{Mach number contours computed with the dual time-stepping BDF1 $p=4$ scheme for the $M_{\infty}=6.85$ SBLI case on the $N_\text{elem}=27,990$ grid.}\label{fig:Ma6-solution}
\end{figure}


\begin{figure}
  \centering
  \begin{subfigure}{0.49\textwidth}
    \centering
    \includegraphics[height=0.245\textheight]{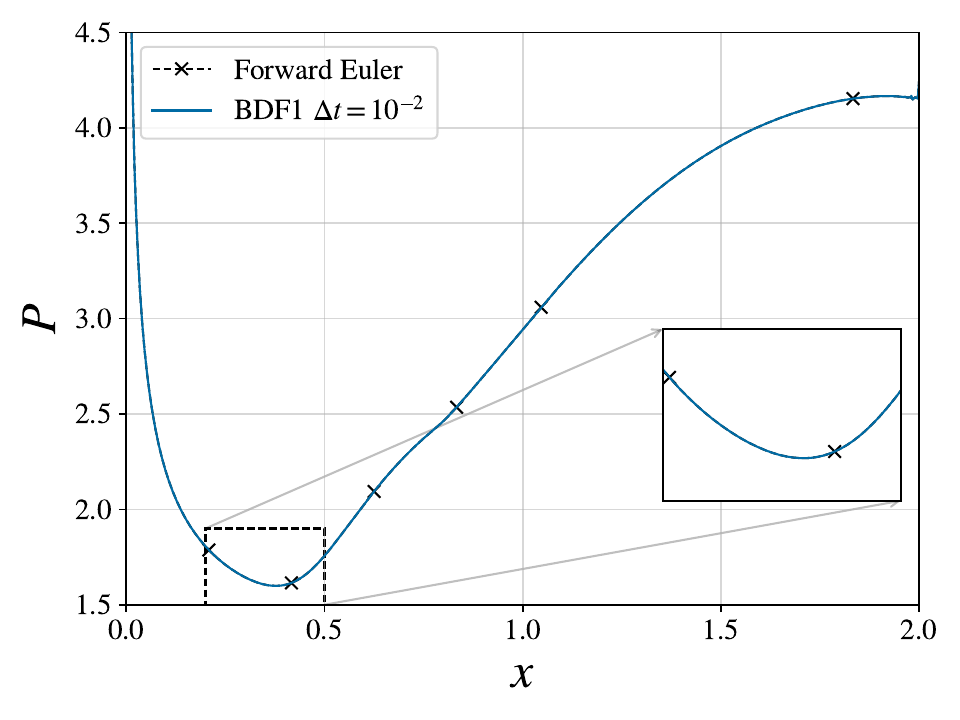}
  \end{subfigure}
  \begin{subfigure}{0.49\textwidth}
    \centering
    \includegraphics[height=0.245\textheight]{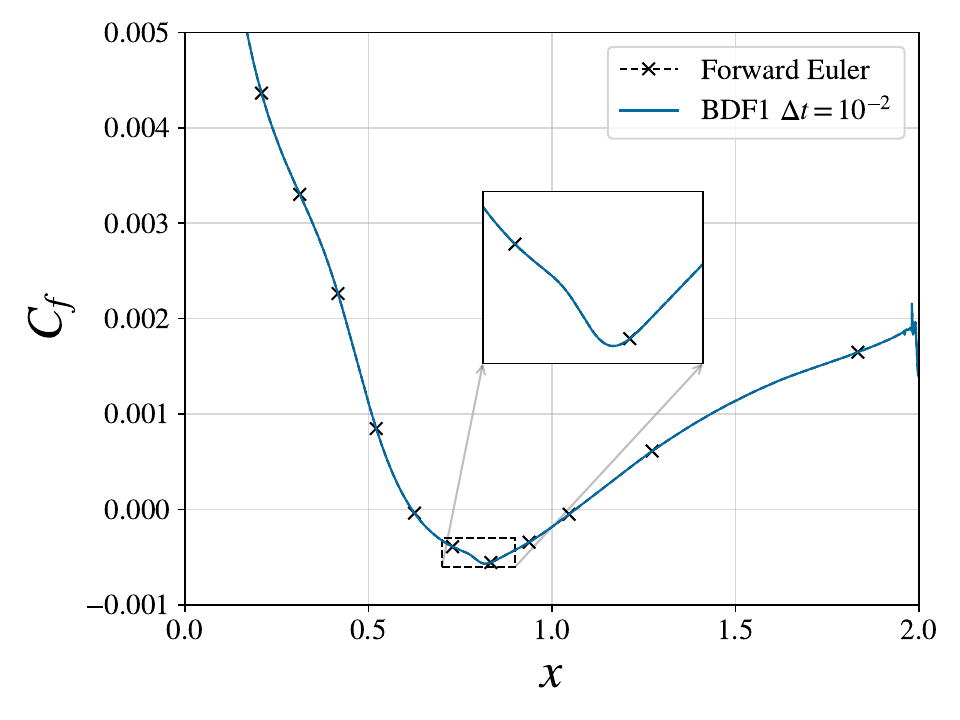}
  \end{subfigure}
  \caption{The wall pressure (left panel) and {the} skin friction coefficient computed with the dual time-stepping BDF1 and forward Euler $p=4$ schemes on the $N_\text{elem}=27,990$ grid for the $M_{\infty}=6.85$ SBLI case.}
  \label{fig:osbli-27k-p6-wall}
\end{figure}


We now solve the SBLI problem at a significantly higher Mach number ($M_{\infty}=6.85$) to see how it affects the accuracy of the proposed dual time-stepping scheme. Figure \ref{fig:Ma6-solution}, shows Mach number contours computed by the BDF1 dual time-stepping $p=4$ scheme for the $M_{\infty}=6.85$ case on the $N_\text{elem}=27,990$ unstructured quadrilateral grid, which are nearly identical to those presented in Fig. 28 in ~\cite{upperman2021high}. 

The wall pressure and {the} skin friction coefficient computed with the BDF1 and forward Euler $p=4$ schemes are shown in Figure~\ref{fig:osbli-27k-p6-wall}, where the point $(0,0)$ corresponds to the flat plate leading edge and $(x,y)=(2,0)$ is the right boundary of the computational domain. This comparison shows that both $P_{w}$ and $C_f$ computed using the dual time-stepping BDF1 scheme are nearly identical to those obtained with the forward Euler method. 

Unlike the forward Euler scheme, the dual time-stepping BDF1 scheme converges nearly monotonically for this test case. Note, however, that both schemes demonstrate similar convergence rates, which are due to the fact that the pseudotime step for this test case is solely bounded by the temperature positivity constraint that is sharp and cannot be relaxed.


\subsection{2D cylinder flow at ${M_{\infty}=17.605}$}


\begin{figure}
  \centering
  \begin{subfigure}{0.49\textwidth}
    \centering
    \includegraphics[height=0.25\textheight]{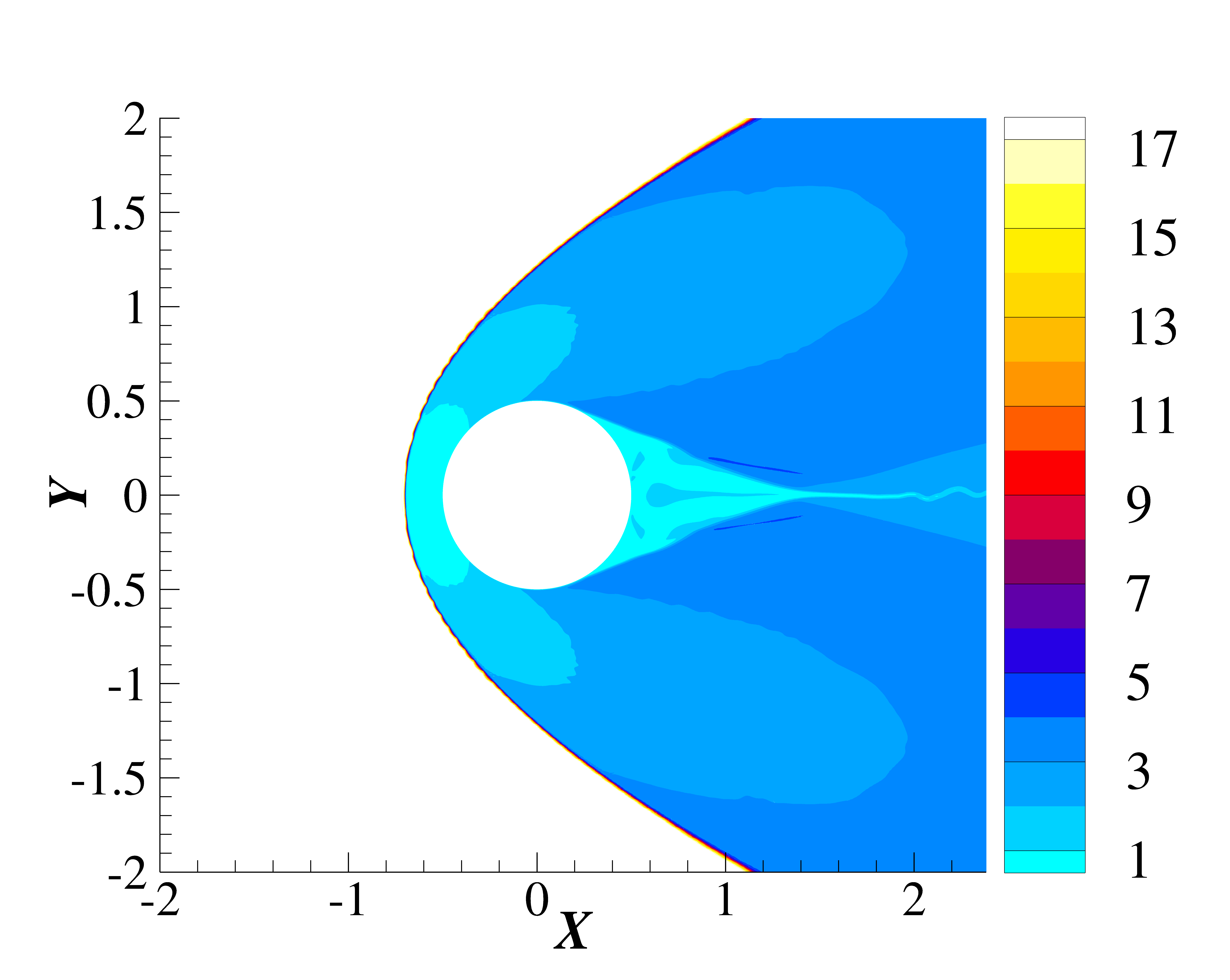}
  \end{subfigure}
  \begin{subfigure}{0.49\textwidth}
    \centering
    \includegraphics[height=0.25\textheight]{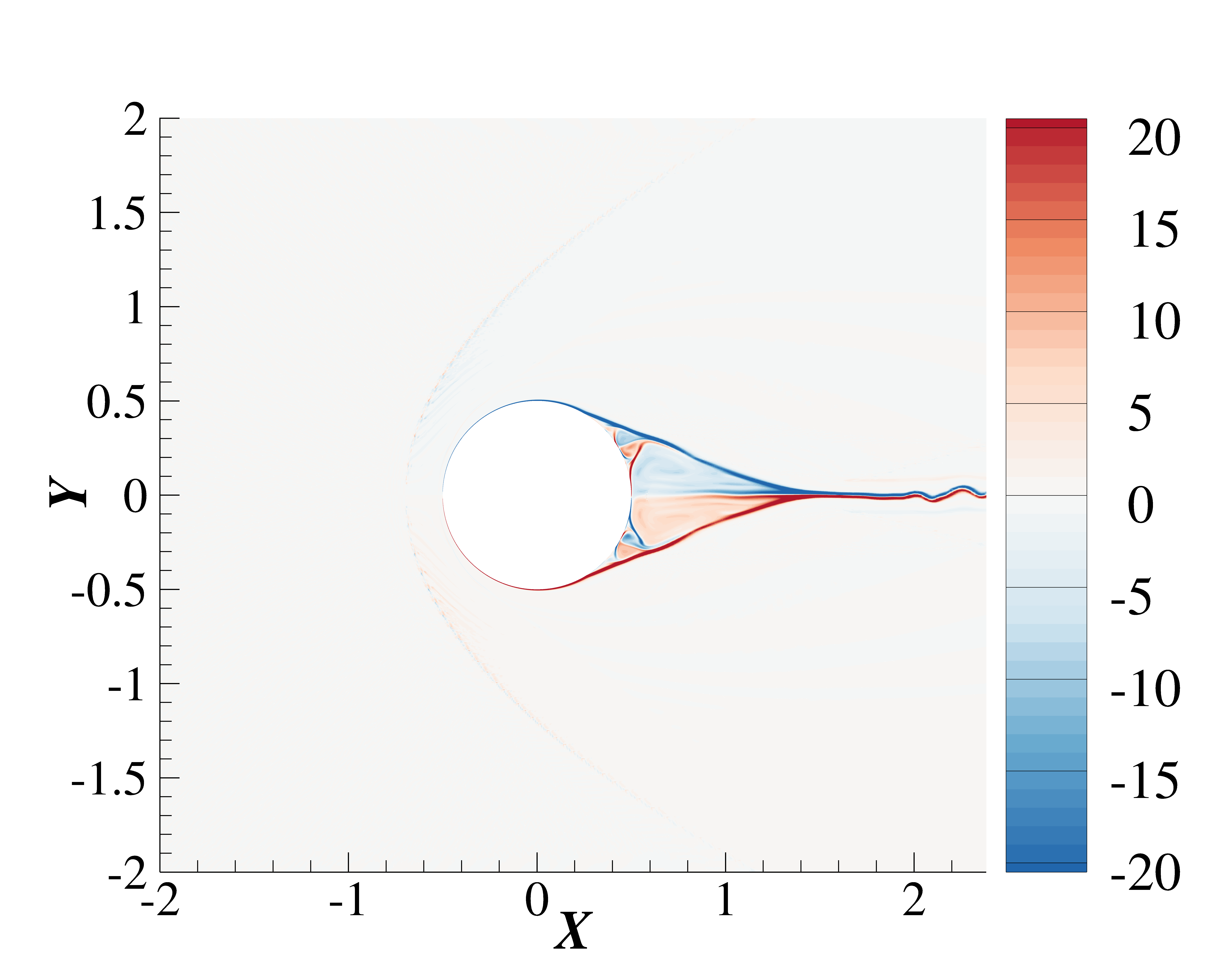}
  \end{subfigure}
  \caption{Mach number (left panel) and vorticity contours computed with the BDF2 dual time-stepping $p=5$ scheme for the hypersonic cylinder flow at $M_{\infty}=17.605$ on the $55,216$ element grid.}
  \label{fig:hypcyl-55k-p5-solution}
\end{figure}


\begin{figure}
  \centering
  \begin{subfigure}{0.49\textwidth}
    \centering
    \includegraphics[height=0.245\textheight]{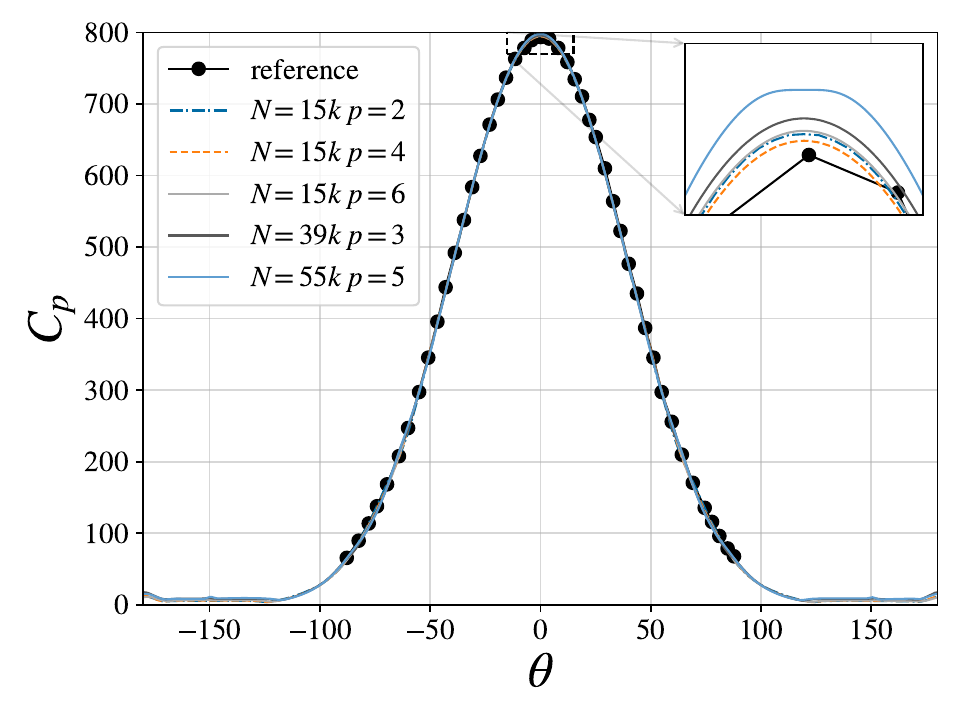}
  \end{subfigure}
  \begin{subfigure}{0.49\textwidth}
    \centering
    \includegraphics[height=0.245\textheight]{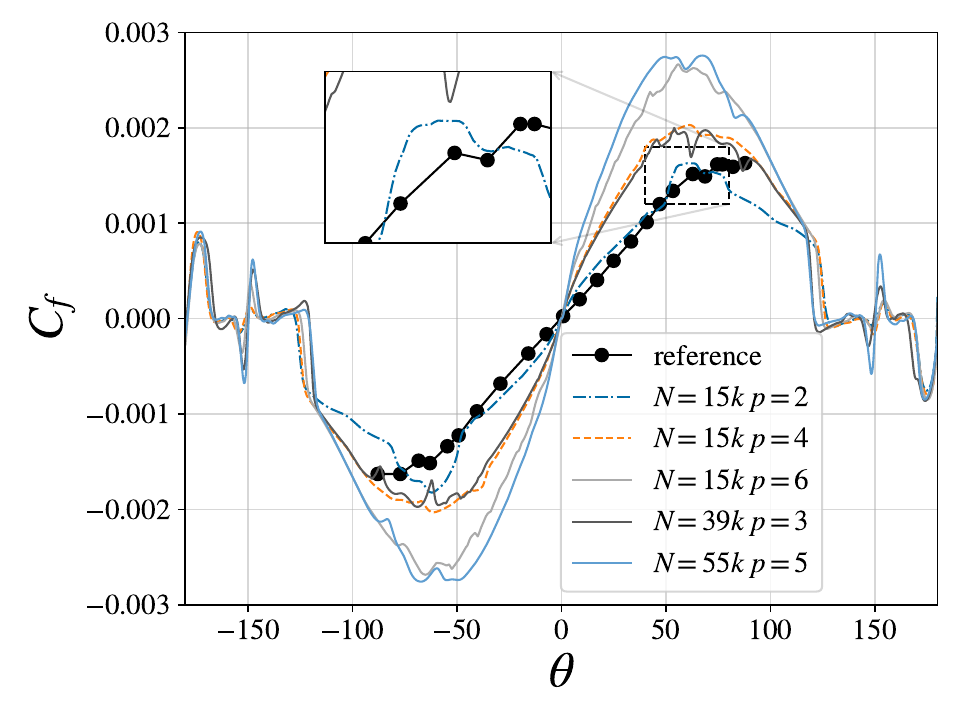}
  \end{subfigure}
  \caption{Time-averaged wall pressure (left panel) and skin friction coefficients computed using the dual time-stepping BDF1 scheme and the reference solution~\cite{fernandez2018physics} for the hypersonic cylinder flow at $M_{\infty}=17.605$ on the $15,216$, $p=6$ and $55,216$, $p=5$ element grids.}
  \label{fig:hypcyl-55k-p5-wall}
\end{figure}


Another test problem, for which the positivity of thermodynamic variables plays a critical role, is the 2D hypersonic flow around a cylinder at $M_{\infty}=17.605$ and $Re_{\infty}=376,930$. The cylinder center is located at $(x,y)=(0,0)$, and its radius is set equal to $r=0.5$. The computational domain is a rectangle: $-2\leq x\leq 3$ and $-2\leq y\leq 2$. The supersonic inflow and outflow boundary conditions are imposed on the left and right boundaries, respectively, while the supersonic freestream boundary conditions are used at the top and bottom boundaries. The entropy{-}stable no-slip boundary conditions similar to those used in the previous test problem are imposed on the cylinder wall. For this test problem, three grids with coarse ($N_{\rm elem}= 15,776$), medium ($N_{\rm elem}= 39,440$), and fine ($N_{\rm elem}=55,216$) resolutions are considered. These grids are stretched in the direction normal to the cylinder wall so that their wall grid spacings are $\Delta r=2.67\times 10^{-3}, 2.0\times 10^{-3},$ and $1.33\times 10^{-3}$, respectively. The flow is initialized with the constant freestream flow, whose velocity vector is gradually reduced to zero at the cylinder wall. The coarse- and medium-grid cases are run using the dual time-stepping BDF1 scheme, while the numerical solution on the fine grid is obtained using the BDF2 counterpart. All test cases are run until $t=40$, corresponding to the time when the wake becomes fully developed, as one can see in Fig.~\ref{fig:hypcyl-55k-p5-solution}. After that, each test case is integrated for additional $5$ nondimensional time units to compute time-averaged quantities.  Since no appreciable difference is observed between the BDF dual time-stepping and SSPRK3 solutions for this test case, only the BDF results are presented herein.

In Fig.~\ref{fig:hypcyl-55k-p5-wall}, the time-averaged wall pressure and skin friction coefficients computed using the proposed first- and second-order BDF dual time-stepping schemes are compared with the reference solution obtained with the 4th-order hybridized Discontinuous Galerkin (DG) method~\cite{fernandez2018physics}. As shown in the figure, the pressure coefficient obtained with the present scheme is practically identical to that of the reference solution for all grids considered. It should be noted, however, that the skin friction coefficient is a much more sensitive quantity. On the coarsest grid, which is comparable with the $N_{\rm elem}=16,000$ grid used in~\cite{fernandez2018physics}, the difference between the present and reference skin friction coefficients is less than $10\%$, which can be explained by the difference in the time averaging windows. This discrepancy increases as the grid is refined, which can be attributed to the lack of grid resolution in the boundary layer provided by the 16,000-element grid used in \cite{fernandez2018physics}.


\subsection{3D supersonic Taylor-Green vortex flow}

The last test problem is the 3D Taylor--Green vortex (TGV) flow at the Reynolds number of $Re_{\infty}=400$ and two different Mach numbers, $M_{\infty} = 2$ and $10$. This problem is considered to test how the proposed dual time-stepping schemes perform for the essentially unsteady turbulent flow with strong discontinuities. The TGV problem is solved on the periodic cube ($-\pi \leq x,y,z \leq \pi$) with the following initial conditions: 
\begin{align}
  \begin{split}\label{eq:tgv-Re400-init}
    p(x,y,z) &= 1 + \frac{\gamma M_{\infty}^2}{C_{M}}(\cos 2x + \cos 2y) (\cos 2z + 2),\\
    v_x(x,y,z) &= \sin x \cos y \cos z,\\
    v_y(x,y,z) &= - \cos x \sin y \cos z,\\
    v_z(x,y,z) &= 0,\\
    T(x,y,z) &= 1,
  \end{split}
\end{align}
where $C_M$ is a constant introduced to make our nondimensionalization consistent with the one used in~\cite{peng2018effects}. The initial density distribution is computed as $\rho(x,y,z) = p/T$. The following quantities are measured to evaluate accuracy and convergence properties of the proposed dual time-stepping scheme: 
\begin{equation}
\nonumber
  E_k = \frac{(\gamma-1)M_{\infty}^2}{|\Omega|}\int_\Omega\frac{\rho v_i v_i}{2} \text{d}\Omega,
\end{equation}
\begin{equation}
\nonumber
  \varepsilon^S = \frac{(\gamma-1)M_{\infty}^2}{Re|\Omega|}\int_\Omega\mu(T)\omega_i \omega_i \text{d}\Omega,
\end{equation}
\begin{equation}
\nonumber
\label{dilation}
  \varepsilon^D = \frac{4(\gamma-1)M_{\infty}^2}{3Re|\Omega|}\int_\Omega\mu(T)\left(\frac{\partial v_d}{\partial x_d}\right)^2 \text{d}\Omega,
\end{equation}
where $E_k$ is a total kinetic energy, $|\Omega|$ is the domain volume, $ \varepsilon^S$ and $\varepsilon^D$ are solenoidal and dilational contributions to the viscous dissipation rate of the kinetic energy.


\subsubsection{${M_{\infty}=2.0}$ case}


\begin{figure}
  \centering
  \begin{subfigure}{0.48\textwidth}
    \centering
    \includegraphics[width=0.99\textwidth]{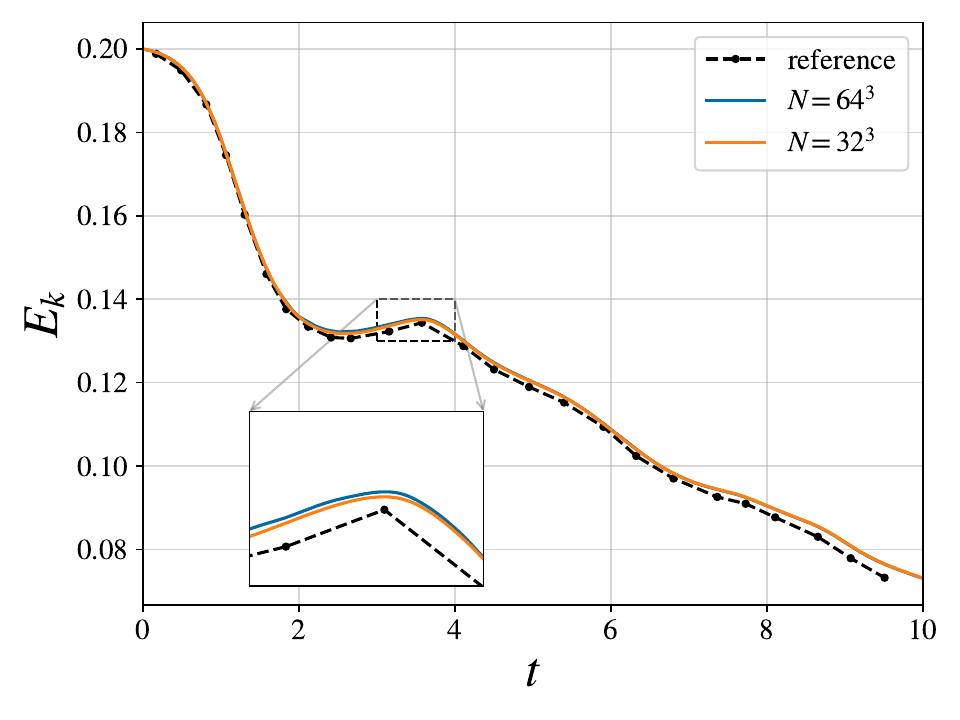}
  \end{subfigure}\hfill
  \begin{subfigure}{0.48\textwidth}
    \centering
    \includegraphics[width=0.99\textwidth]{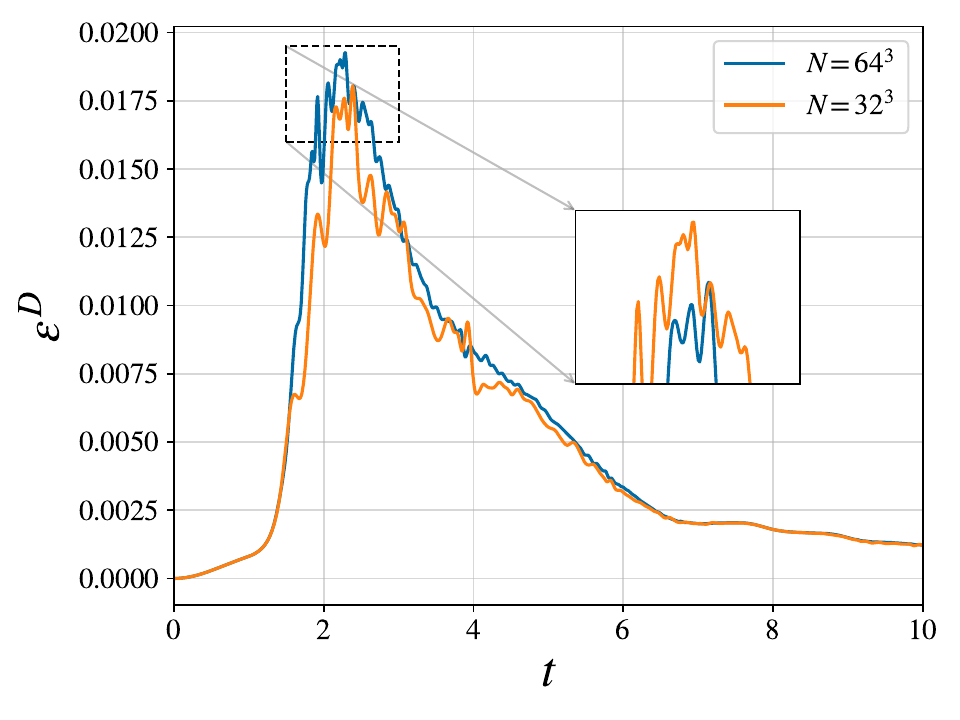}
  \end{subfigure}
  \caption{Time histories of the kinetic energy (left-panel) and the dilation computed with the dual time-stepping BDF2 scheme on the grid with $\Delta t=10^{-3}$ and $N_\text{elem}= 32^3, 64^3$ for the TGV flow at $Re_{\infty}=400$ and $M_{\infty}=2.0$.}
  \label{fig:tgv-Re400-Ma2}
\end{figure}


\begin{table}
  \centering
  \caption{The time-averaged kinetic energy dissipation rates. The reference quantities are taken from~\cite{peng2018effects}.}\label{tab:tgv-kedr-Ma2-Re400}
  \begin{tabular}{l|c}
    $N_\text{elem}$ & $\bar{\varepsilon}^V$\\
    \hline
    $32^3$ & -- \\
    $64^3$ & -- \\
    $128^3$\cite{peng2018effects} & $0.0054$\\
    $256^3$\cite{peng2018effects} & $0.0060$\\
    $512^3$\cite{peng2018effects} & $0.0063$\\
  \end{tabular}\hspace{1em}
  \begin{tabular}{l|ccccc}
    $N_\text{elem}$ & $\bar{\varepsilon}^V$ & $\bar{\varepsilon}^S$ & $\bar{\varepsilon}^D$ \\
    \hline
    $32^3$ & $0.0067$ & $0.0036$ & $0.0030$ \\
    $64^3$ & $0.0069$ & $0.0036$ & $0.0033$ \\
    $128^3$ & --\\
    $256^3$ & --\\
    $512^3$ & --\\
  \end{tabular}\vspace{1em}
\end{table}


First,  we consider the $M_{\infty}=2$ case, for which the parameter $C_{M}$ in Eq.~(\ref{eq:tgv-Re400-init}) is set equal to $89.6$ to match the initial condition in~\cite{peng2018effects}. Time histories of the total kinetic energy (left panel) and the dilational component, $ \varepsilon^D$, computed using the new dual time-stepping BDF2 method on $N_\text{elem}=32^3,64^3$ grids are presented in Figure \ref{fig:tgv-Re400-Ma2}. As follows from this comparison, the present kinetic energy profiles are in excellent agreement with each other and {the} reference solution obtained with the 7/8th-order hybrid Weighted-Essentially Non-Oscillatory (WENO) finite difference scheme on the $N_\text{elem}=512^3$ grid, which is reported in~\cite{peng2018effects}. For the dilational component of the dissipation rate of the kinetic energy shown in Fig.~\ref{fig:tgv-Re400-Ma2}, the discrepancy between the BDF2 solutions on the medium and fine grids is more pronounced for $2\le t \le 3$, thus indicating that the $N_\text{elem}=32^3$ grid resolution is not enough to accurately predict gradient quantities such as $\varepsilon^D$ for this test problem.  We also compare the time-averaged viscous dissipation rate of the kinetic energy and its components with the corresponding quantities presented in~\cite{peng2018effects}. As can be seen, in Table \ref{tab:tgv-kedr-Ma2-Re400},  our results are consistent with each other, i.e.,  $\bar{\varepsilon}^V=\bar{\varepsilon}^S+\bar{\varepsilon}^D$, and agree well with those obtained in ~\cite{peng2018effects} with much larger number of degrees of freedom. 


\subsubsection{${M_{\infty}=10}$ case}


\begin{figure}
  \centering
  \begin{subfigure}{0.49\textwidth}
    \centering
    \includegraphics[height=0.245\textheight]{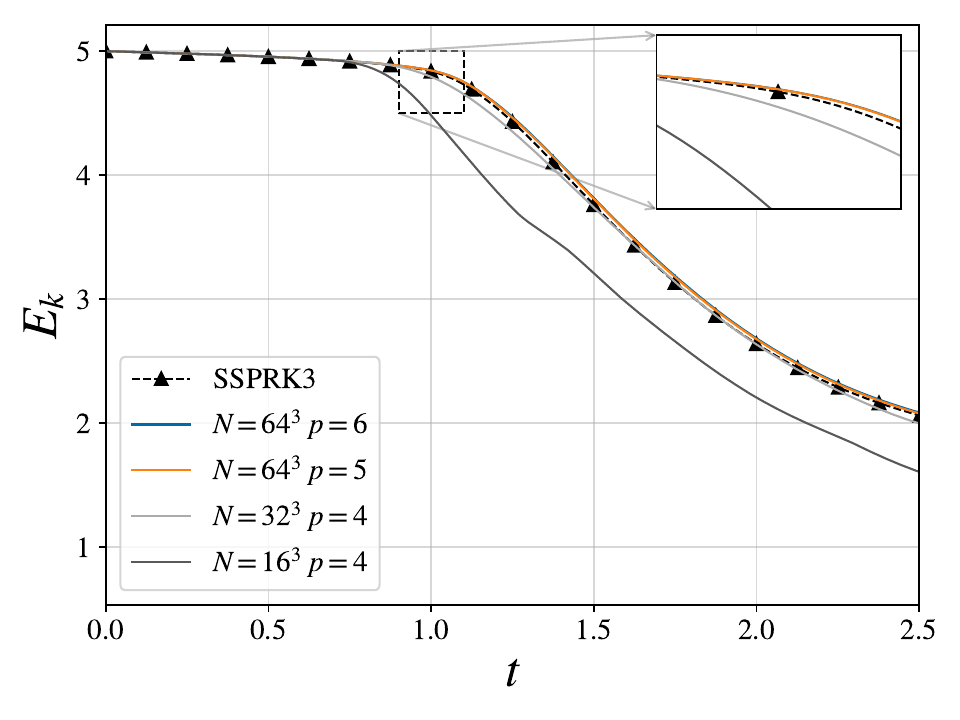}
  \end{subfigure}\hfill
  \begin{subfigure}{0.49\textwidth}
    \centering
    \includegraphics[height=0.245\textheight]{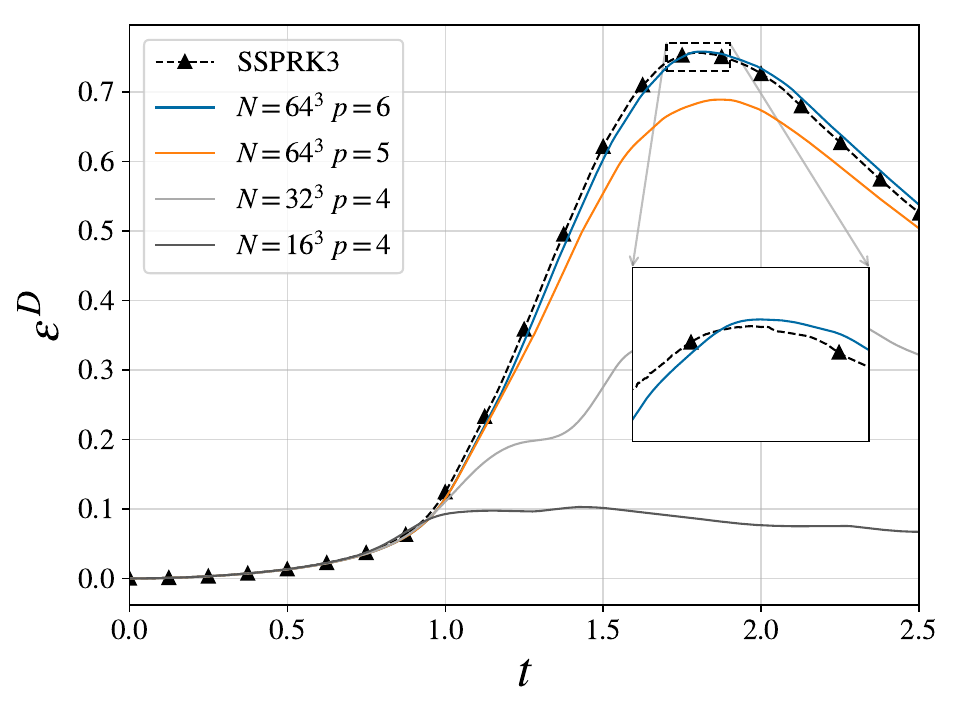}
  \end{subfigure}
  \caption{Time histories of the total kinetic energy (left-panel) and dilation computed with the dual time-stepping BDF2 $p=4,5,6$ and SSPRK3 $p=6$ schemes on $N_\text{elem}=16^3,32^3,64^3$ and $N_\text{elem}=64^3$ grids, respectively, for the TGV flow at $Re_{\infty}=400$ and $M_{\infty}=10$.}
  \label{fig:tgv-Re400-Ma10}
\end{figure}


\begin{figure}
  \centering
  \begin{minipage}{0.49\textwidth}
    \begin{subfigure}{\textwidth}
      \centering
      \includegraphics[height=0.245\textheight]{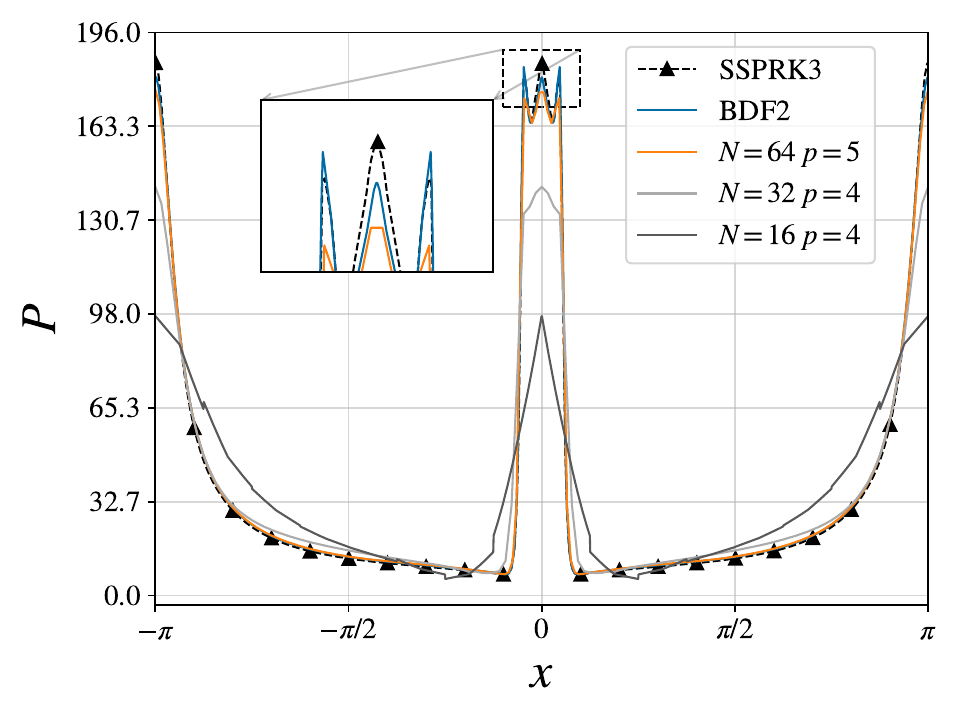}
    \end{subfigure}
  \end{minipage}
  \hfill
  \begin{minipage}{0.49\textwidth}
    \begin{subfigure}{\textwidth}
      \centering
      \includegraphics[height=0.245\textheight]{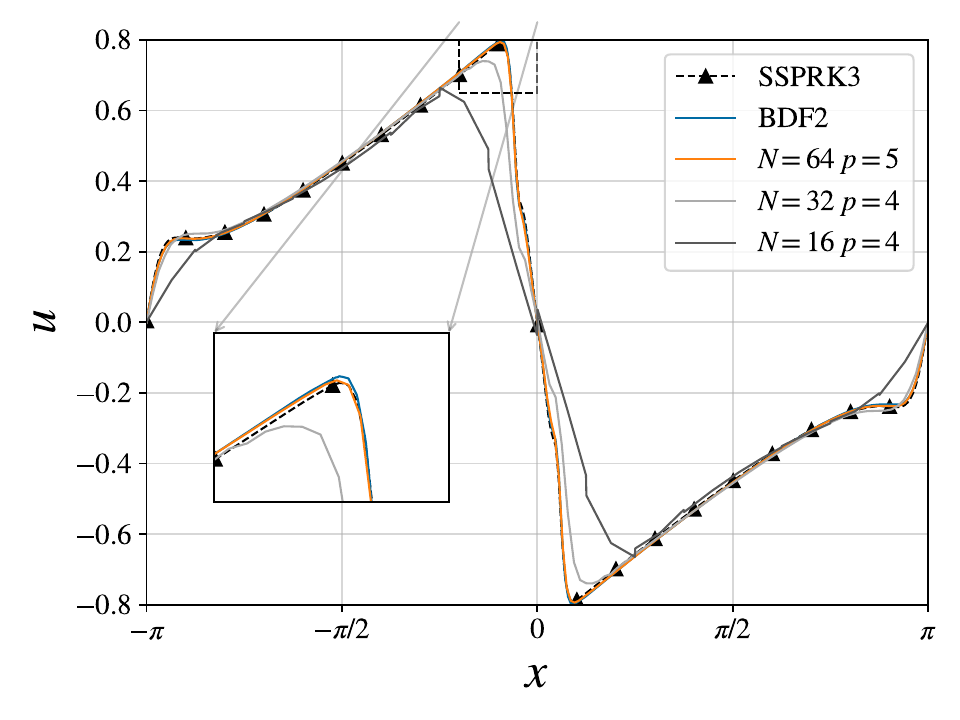}
    \end{subfigure}
  \end{minipage}\\
  \caption{Pressure (left-panel) and $x$-component of the velocity vector profiles along the line $y=\pi$ and $z=0$ obtained with the dual time-stepping BDF2 $N_\text{elem}=16^3,32^3,64^3$ $p=4,5,6$ (the legend BDF2 refers to the $N_\text{elem}=64^3$ $p=6$ case) and SSPRK3 $N_\text{elem}=64^3$ $p=6$ schemes for the TGV flow at $Re_{\infty}=400$ and $M_{\infty}=10$.}
  \label{fig:tgv-Re400-Ma10-primitives}
\end{figure}


No reference solutions or experimental data are available in the literature for the $M_{\infty}=10$ case. Therefore, we compare the BDF2 dual time-stepping results against those obtained with the SSPRK3 counterpart with different polynomial {degrees} and grid resolutions: $(p, N_{\rm elem}) = (4, 16^3), (4, 32^3), (5, 64^3), $ $(6, 64^3)$. Figure~\ref{fig:tgv-Re400-Ma10} shows time histories of the total kinetic energy and the dilational component of the kinetic energy dissipation rate computed with the SSPRK3 and dual time-stepping BDF2 schemes. This comparison shows that the BDF2 $p=4$ solution demonstrates convergence as the grid is refined. Furthermore, the results obtained with the BDF2 and SSPRK3 $p=6$ schemes on the finest $N_\text{elem}=64^3$ grid are nearly identical, as seen in Fig.~\ref{fig:tgv-Re400-Ma10}. This comparison shows that the present dual time-stepping BDF2 scheme provides {excellent} temporal accuracy for gradient quantities such as the dilational component, which is very sensitive to strong shock waves and their interaction with vortices.  We also compare snapshots of the pressure and $x$-component of the velocity vector profiles along the line $(x,y,z)=(x,\pi,0)$, which are shown in Fig.~\ref{fig:tgv-Re400-Ma10-primitives}.  As one can see in this figure, the BDF2 solution is slightly more dissipative than that computed with the SSPRK3 scheme on the same $p=6$, $N_\text{elem}=64^3$ grid. This can be explained by the fact that on average, the physical time step of the BDF2 scheme is $20$ times higher than that of the explicit SSPRK3 scheme. It should also be noted that for this TGV flow, the dual time-stepping BDF2 scheme provides a nearly $30\%$ speed-up in the wall{-}clock time as compared with the SSPRK3 scheme with the same $p=6$ discretization and grid resolution $N_\text{elem}=64^3$ in space.


\section{Conclusions}
\label{sec:conclusions}


In this paper,  we have developed {the} new implicit BDF1 and BDF2 dual time-stepping positivity-preserving entropy{-}stable spectral collocation schemes of arbitrary spatial order of accuracy for solving the 3D compressible Navier{--}Stokes equations. To solve the nonlinear system of discrete equations resulting from the implicit BDF discretization at each physical time step, a dual time-stepping technique with the explicit forward Euler discretization in the pseudotime is used. This dual time-stepping strategy allows us to combine unconditional stability properties of the implicit 1st- and 2nd-order BDF time integrators with the positivity-preserving and entropy stability properties of the baseline explicit spectral collocation scheme, while providing the design-order of accuracy in the physical time. The key distinctive feature of the proposed methodology is that it guarantees the positivity of thermodynamic variables at each pseudotime iteration and imposes no constraints on the physical time step size. Another distinct property of this dual time-stepping scheme is the fact that the upper bound on the pseudotime step required for {the} positivity of density is greater than that of the original explicit Euler scheme in the physical time, thus allowing us to increase the overall efficiency of the method as compared with the explicit counterpart for most test cases considered.


Though the unconditional stability properties of the implicit BDF1 and BDF2 schemes considered in this study are not jeopardized, the new dual time-stepping positivity-preserving entropy{-}stable schemes are still bounded by the explicit stability conditions in the pseudotime. To overcome this CFL-type constraint in the pseudotime and use the unconditional stability of the proposed implicit positivity-preserving entropy{-}stable schemes to their full capacity, positivity-preserving nonlinear solvers are needed, which is the subject of ongoing research.


\section*{Funding}
The second author was supported by the Department of Defense under Grant No. W911NF2310183.


\section*{Data availability}
The datasets generated and/or analyzed during the current study are available from the corresponding author on reasonable request.



\bibliographystyle{elsarticle-num-names} 
\bibliography{references}

@article{upperman2023first,
  author =        {Upperman, Johnathon and Yamaleev, Nail K},
  journal =       {Journal of Scientific Computing},
  number =        {1},
  pages =         {18},
  publisher =     {Springer},
  title =         {First-Order Positivity-Preserving Entropy Stable
                   Scheme for the 3-{D} Compressible {N}avier--{S}tokes
                   Equations},
  volume =        {94},
  year =          {2023},
}

@article{yamaleev2023high,
  author =        {Yamaleev, Nail K and Upperman, Johnathon},
  journal =       {Journal of Scientific Computing},
  number =        {1},
  pages =         {11},
  publisher =     {Springer},
  title =         {High-order positivity-preserving entropy stable
                   schemes for the 3-{D} compressible {N}avier--{S}tokes
                   equations},
  volume =        {95},
  year =          {2023},
}

@inproceedings{shu2018bound,
  address =       {Cham},
  author =        {Shu, Chi-Wang},
  booktitle =     {Theory, Numerics and Applications of Hyperbolic
                   Problems II},
  pages =         {591--603},
  publisher =     {Springer International Publishing},
  title =         {Bound-Preserving High-Order Schemes for Hyperbolic
                   Equations: Survey and Recent Developments},
  year =          {2018},
  isbn =          {978-3-319-91548-7},
}

@article{zhang2011survey,
  author =        {Zhang, Xiangxiong and Shu, Chi-Wang},
  journal =       {Proc. R. Soc. A.},
  month =         oct,
  number =        {2134},
  pages =         {2752--2776},
  title =         {Maximum-principle-satisfying and
                   positivity-preserving high-order schemes for
                   conservation laws: survey and new developments},
  volume =        {467},
  year =          {2011},
  doi =           {10.1098/rspa.2011.0153},
  issn =          {1364-5021, 1471-2946},
  language =      {en},
  url =           {https://royalsocietypublishing.org/doi/10.1098/
                  rspa.2011.0153},
}

@article{zhang2017positivity,
  author =        {Zhang, Xiangxiong},
  journal =       {Journal of Computational Physics},
  pages =         {301--343},
  publisher =     {Elsevier},
  title =         {On positivity-preserving high order discontinuous
                   {G}alerkin schemes for compressible
                   {N}avier--{S}tokes equations},
  volume =        {328},
  year =          {2017},
}

@article{guermond2021second,
  author =        {Guermond, Jean-Luc and Maier, Matthias and
                   Popov, Bojan and Tomas, Ignacio},
  journal =       {Computer Methods in Applied Mechanics and
                   Engineering},
  pages =         {113608},
  publisher =     {Elsevier},
  title =         {Second-order invariant domain preserving
                   approximation of the compressible {N}avier--{S}tokes
                   equations},
  volume =        {375},
  year =          {2021},
}

@phdthesis{upperman2021high,
  author =        {Upperman, Johnathon Keith},
  school =        {Old Dominion University},
  title =         {High-Order Positivity-Preserving $L_2$-Stable
                   Spectral Collocation Schemes for the 3-{D}
                   Compressible {N}avier-{S}tokes Equations},
  year =          {2021},
}

@article{upperman2022positivity,
  author =        {Upperman, Johnathon and Yamaleev, Nail K},
  journal =       {Journal of Computational Physics},
  pages =         {111355},
  publisher =     {Elsevier},
  title =         {Positivity-preserving entropy stable schemes for the
                   1-{D} compressible {N}avier-{S}tokes equations:
                   First-order approximation},
  volume =        {466},
  year =          {2022},
}

@article{yamaleev2022positivity,
  author =        {Yamaleev, Nail K and Upperman, Johnathon},
  journal =       {J. Comput. Phys.},
  pages =         {111354},
  publisher =     {Elsevier},
  title =         {Positivity-preserving entropy stable schemes for the
                   1-{D} compressible {N}avier-{S}tokes equations:
                   High-order flux limiting},
  volume =        {466},
  year =          {2022},
}

@article{lin2023positivity,
  author =        {Lin, Yimin and Chan, Jesse and Tomas, Ignacio},
  journal =       {Journal of Computational Physics},
  pages =         {111850},
  publisher =     {Elsevier},
  title =         {A positivity preserving strategy for entropy stable
                   discontinuous {G}alerkin discretizations of the
                   compressible {E}uler and {N}avier--{S}tokes
                   equations},
  volume =        {475},
  year =          {2023},
}

@inproceedings{belov1995new,
  address =       {Reno,NV,U.S.A.},
  author =        {Belov, A and Martinelli, L and Jameson, A},
  booktitle =     {33rd {Aerospace} {Sciences} {Meeting} and {Exhibit}},
  month =         jan,
  publisher =     {American Institute of Aeronautics and Astronautics},
  title =         {A new implicit algorithm with multigrid for unsteady
                   incompressible flow calculations},
  year =          {1995},
  doi =           {10.2514/6.1995-49},
  language =      {en},
  url =           {https://arc.aiaa.org/doi/10.2514/6.1995-49},
}

@article{Knoll1998,
  author =        {Knoll, D.A. and McHugh, P.R},
  journal =       {SIAM J. of Sci. Comput.},
  number =        {1},
  pages =         {291-301},
  title =         {Enhanced nonlinear iterative techniques applied to a
                   nonequilibrium plasma flow},
  volume =        {19},
  year =          {1998},
}

@inproceedings{pandya2003implementation,
  address =       {Reno, Nevada},
  author =        {Pandya, Shishir and Venkateswaran, Sankaran and
                   Pulliam, Thomas},
  booktitle =     {41st {Aerospace} {Sciences} {Meeting} and {Exhibit}},
  month =         jan,
  publisher =     {American Institute of Aeronautics and Astronautics},
  title =         {Implementation of {Preconditioned} {Dual}-{Time}
                   {Procedures} in {OVERFLOW}},
  year =          {2003},
  doi =           {10.2514/6.2003-72},
  isbn =          {978-1-62410-099-4},
  language =      {en},
  url =           {https://arc.aiaa.org/doi/10.2514/6.2003-72},
}

@inproceedings{nakashima2014development,
  author =        {Nakashima, Yoshitaka and Watanabe, Norihiko and
                   Nishikawa, H},
  booktitle =     {The 28th Computational Fluid Dynamics Symposium},
  organization =  {Japan Soc. of Fluid Dynamics},
  title =         {Development of an effective implicit solver for
                   general-purpose unstructured {CFD} software},
  year =          {2014},
}

@article{nordstrom2019dual,
  author =        {Nordstr{\"o}m, Jan and Ruggiu, Andrea A},
  journal =       {Journal of Scientific Computing},
  pages =         {1050--1071},
  publisher =     {Springer},
  title =         {Dual time-stepping using second derivatives},
  volume =        {81},
  year =          {2019},
}

@inproceedings{arnone1993multigrid,
  address =       {Orlando,FL,U.S.A.},
  author =        {Arnone, Andrea and Liou, Meng-Sing and
                   Povinelli, Louis},
  booktitle =     {11th {Computational} {Fluid} {Dynamics} {Conference}},
  month =         jul,
  publisher =     {American Institute of Aeronautics and Astronautics},
  title =         {Multigrid time-accurate integration of
                   {{N}avier--{S}tokes} equations},
  year =          {1993},
  doi =           {10.2514/6.1993-3361},
  language =      {en},
  url =           {https://arc.aiaa.org/doi/10.2514/6.1993-3361},
}

@article{brenner2005navier,
  author =        {Brenner, Howard},
  journal =       {Physica A: Statistical Mechanics and its
                   Applications},
  number =        {1-2},
  pages =         {60--132},
  publisher =     {Elsevier},
  title =         {{N}avier--{S}tokes revisited},
  volume =        {349},
  year =          {2005},
}

@article{feireisl2010new,
  author =        {Feireisl, Eduard and Vasseur, Alexis},
  journal =       {New Directions in Mathematical Fluid Mechanics: The
                   Alexander V. Kazhikhov Memorial Volume},
  pages =         {153--179},
  publisher =     {Springer},
  title =         {New perspectives in fluid dynamics: Mathematical
                   analysis of a model proposed by {H}oward {B}renner},
  year =          {2010},
}

@article{guermond2014viscous,
  author =        {Guermond, Jean-Luc and Popov, Bojan},
  journal =       {SIAM Journal on Applied Mathematics},
  number =        {2},
  pages =         {284--305},
  publisher =     {SIAM},
  title =         {Viscous regularization of the {E}uler equations and
                   entropy principles},
  volume =        {74},
  year =          {2014},
}

@article{yamaleev2019entropy,
  author =        {Yamaleev, Nail K and Fernandez, David C Del Rey and
                   Lou, Jialin and Carpenter, Mark H},
  journal =       {Journal of Computational Physics},
  pages =         {108897},
  publisher =     {Elsevier},
  title =         {Entropy stable spectral collocation schemes for the
                   3-{D} {N}avier-{S}tokes equations on dynamic
                   unstructured grids},
  volume =        {399},
  year =          {2019},
}

@article{carpenter2014entropy,
  author =        {Carpenter, Mark H and Fisher, Travis C and
                   Nielsen, Eric J and Frankel, Steven H},
  journal =       {SIAM Journal on Scientific Computing},
  number =        {5},
  pages =         {B835--B867},
  publisher =     {SIAM},
  title =         {Entropy stable spectral collocation schemes for the
                   {N}avier--{S}tokes equations: Discontinuous
                   interfaces},
  volume =        {36},
  year =          {2014},
}

@article{fernandez2014review,
  author =        {Fern{\'a}ndez, David C Del Rey and Hicken, Jason E and
                   Zingg, David W},
  journal =       {Computers \& Fluids},
  pages =         {171--196},
  publisher =     {Elsevier},
  title =         {Review of summation-by-parts operators with
                   simultaneous approximation terms for the numerical
                   solution of partial differential equations},
  volume =        {95},
  year =          {2014},
}

@article{fisher2013discretely,
  author =        {Fisher, Travis C and Carpenter, Mark H and
                   Nordstr{\"o}m, Jan and Yamaleev, Nail K and
                   Swanson, Charles},
  journal =       {Journal of Computational Physics},
  pages =         {353--375},
  publisher =     {Elsevier},
  title =         {Discretely conservative finite-difference
                   formulations for nonlinear conservation laws in split
                   form: Theory and boundary conditions},
  volume =        {234},
  year =          {2013},
}

@article{fisher2011boundary,
  author =        {Fisher, Travis C and Carpenter, Mark H and
                   Yamaleev, Nail K and Frankel, Steven H},
  journal =       {Journal of Computational Physics},
  number =        {10},
  pages =         {3727--3752},
  publisher =     {Elsevier},
  title =         {Boundary closures for fourth-order energy stable
                   weighted essentially non-oscillatory
                   finite-difference schemes},
  volume =        {230},
  year =          {2011},
}

@inproceedings{carpenter2016entropy,
  author =        {M. H. Carpenter and T. Fisher and E. Nielsen and
                   M. Parsani and M. Svard and N. Yamaleev},
  booktitle =     {Handbook on Numerical Analysis},
  editor =        {R. Abgrall and C.-W. Shu},
  pages =         {495--524},
  publisher =     {Elsevier},
  title =         {Entropy stable summation-by-parts formulations for
                   compressible computational fluid dynamics},
  year =          {2016},
}

@article{chandrashekar2013kinetic,
  author =        {Chandrashekar, Praveen},
  journal =       {Communications in Computational Physics},
  number =        {5},
  pages =         {1252--1286},
  publisher =     {Cambridge University Press},
  title =         {Kinetic energy preserving and entropy stable finite
                   volume schemes for compressible {E}uler and
                   {N}avier--{S}tokes equations},
  volume =        {14},
  year =          {2013},
}

@article{thomas1979geometric,
  author =        {Thomas, Paul Dennis and Lombard, Charles K},
  journal =       {AIAA journal},
  number =        {10},
  pages =         {1030--1037},
  title =         {Geometric conservation law and its application to
                   flow computations on moving grids},
  volume =        {17},
  year =          {1979},
}

@article{cockburn1998local,
  author =        {B. Cockburn and C.-W. Shu},
  journal =       {SIAM J. of Numer. Anal.},
  number =        {1},
  pages =         {2440-2463},
  title =         {The local discontinuous {G}alerkin method for
                   time-dependent convection-diffusion systems},
  volume =        {35},
  year =          {1998},
}

@inproceedings{merkle1987time,
  address =       {Honolulu,HI,U.S.A.},
  author =        {Merkle, Charles},
  booktitle =     {8th {Computational} {Fluid} {Dynamics} {Conference}},
  month =         jun,
  publisher =     {American Institute of Aeronautics and Astronautics},
  title =         {Time-accurate unsteady incompressible flow algorithms
                   based on artificial compressibility},
  year =          {1987},
  doi =           {10.2514/6.1987-1137},
  language =      {en},
  url =           {https://arc.aiaa.org/doi/10.2514/6.1987-1137},
}

@article{merriam1989entropy,
  author =        {M. {L}. Merriam},
  journal =       {NASA-TM--1989--101086},
  title =         {An Entropy-Based Approach to Nonlinear Stability},
  year =          {1989},
}

@article{tadmor2003entropy,
  author =        {Tadmor, Eitan},
  journal =       {Acta Numerica},
  pages =         {451--512},
  publisher =     {Cambridge University Press},
  title =         {Entropy stability theory for difference
                   approximations of nonlinear conservation laws and
                   related time-dependent problems},
  volume =        {12},
  year =          {2003},
}

@article{dalcin2019conservative,
  author =        {L. Dalcin and D. Rojas and
                   S. Zampini, D. C. D. R. Fernandez and M. H. Carpenter and
                   M. Parsani},
  journal =       {J. of Comput. Phys.},
  title =         {Conservative and entropy stable solid wall boundary
                   conditions for the compressible {N}avier--{S}tokes
                   equations: Adiabatic wall and heat entropy transfer},
  volume =        {397},
  year =          {2019},
}

@misc{blanchard2016sbli,
  author =        {Raphaël Blanchard and Florent Renac},
  howpublished =  {Proc. of the 4th International Workshop on High-Order
                   CFD Methods},
  month =         jun,
  note =          {Crete Island, Greece},
  organization =  {ONERA - CFD Department},
  title =         {Case {B}{L}2: Shock Wave / Laminar Boundary Layer
                   Interaction},
  year =          {2016},
}

@inproceedings{fernandez2018physics,
  address =       {Kissimmee, Florida},
  author =        {Fernandez, Pablo and Nguyen, Cuong and
                   Peraire, Jaime},
  booktitle =     {2018 {AIAA} {Aerospace} {Sciences} {Meeting}},
  month =         jan,
  publisher =     {American Institute of Aeronautics and Astronautics},
  title =         {A physics-based shock capturing method for unsteady
                   laminar and turbulent flows},
  year =          {2018},
  doi =           {10.2514/6.2018-0062},
  isbn =          {978-1-62410-524-1},
  language =      {en},
  url =           {https://arc.aiaa.org/doi/10.2514/6.2018-0062},
}

@article{peng2018effects,
  author =        {Peng, Naifu and Yang, Yue},
  journal =       {Physical Review Fluids},
  number =        {1},
  pages =         {013401},
  publisher =     {APS},
  title =         {Effects of the {M}ach number on the evolution of
                   vortex-surface fields in compressible
                   {T}aylor-{G}reen flows},
  volume =        {3},
  year =          {2018},
}

@Book{GKS,
	author = {Sigal Gottlieb and David Ketcheson and Chi-Wang Shu},
	Title = {Strong Stability Preserving Runge–Kutta and Multistep Time Discretizations},
      publisher = {CWorld Scientific Press},
	year = 2011
}






\end{document}